\documentclass[11pt]{amsart}
\usepackage[T1]{fontenc}
\usepackage[utf8]{inputenc}
\usepackage{lmodern}
\usepackage{amsmath,amssymb,mathtools}
\usepackage{tikz-cd}
\usepackage[expansion=false]{microtype}
\usepackage[margin=1.05in]{geometry}
\usepackage[hidelinks]{hyperref}
\hypersetup{pdfsubject={Weighted jets, rational curves, and the Miyaoka--Mori criterion}}
\numberwithin{equation}{section}
\newtheorem{theorem}{Theorem}[section]
\newtheorem{proposition}[theorem]{Proposition}
\newtheorem{lemma}[theorem]{Lemma}
\newtheorem{corollary}[theorem]{Corollary}
\theoremstyle{definition}
\newtheorem{definition}[theorem]{Definition}
\newtheorem*{remark}{Remark}
\newcommand{\PP}{\mathbb P}
\newcommand{\OO}{\mathcal O}
\newcommand{\Pline}{\mathbb P_{\mathrm{lines}}}
\newcommand{\NE}{\overline{\operatorname{NE}}}
\newcommand{\proofheading}[1]{%
  \par\noindent{\normalfont\bfseries #1}\enspace\ignorespaces
}
\DeclareMathOperator{\Tot}{Tot}
\DeclareMathOperator{\Spec}{Spec}
\DeclareMathOperator{\Proj}{Proj}
\DeclareMathOperator{\ord}{ord}
\begin{document}
\title[CONSTRUCTING RATIONAL CURVES VIA JETS]{Constructing Rational Curves via Jets on Projective Varieties with Non-Nef Canonical Bundle}
\author{Bin Guo}
\address{Academy of Mathematics and Systems Science, Chinese Academy of Sciences, Beijing 100190, China}
\email{guobin@amss.ac.cn}
\author{Song-Yan Xie}
\address{State Key Laboratory of Mathematical Sciences, Academy of Mathematics and Systems Science, Chinese Academy of Sciences, Beijing 100190, China; School of Mathematical Sciences, University of Chinese Academy of Sciences, Beijing 100049, China}
\email{xiesongyan@amss.ac.cn}
\date{September 23, 2026}
\subjclass[2020]{Primary 14E30; Secondary 14J45, 14C17, 14D15}

\begin{abstract}
We give an algebraic proof in characteristic zero of the Miyaoka--Mori
criterion: every point of a curve of negative canonical degree on a smooth
projective variety lies on a rational curve. Our jet technique gives, in
addition, an effective numerical decomposition of the original curve
class. For each prescribed point, the decomposition contains a rational
curve through that point, with a positive coefficient independent of the
point and with anticanonical degree at most $\dim X+1$. Together with
BDPP cone duality, this recovers the projective uniruledness criterion
over $\mathbb C$.

The main result of this paper was obtained using the Pharos system. A detailed
report on the use of Pharos and on the Lean~4 formalization of the main
arguments is given in the Appendix B, written by Bin Dong, Guoxiong Gao,
Zeming Sun, and Bin Wu.
\end{abstract}
\maketitle

\section{Introduction}\label{sec:introduction}

The criterion of Miyaoka and Mori~\cite{MM86} produces rational
curves through every point of a curve of negative canonical degree. Its
classical proof uses Frobenius in positive characteristic
to amplify the anticanonical degree without changing the genus, then
applies Mori's bend-and-break~\cite{Mor79}. Boundedness and
specialization bring the rational curves back to characteristic zero;
see~\cite[Proposition~7.1, \S7.2, and Theorem~7.7]{Deb}
and~\cite[II, Theorem~5.8]{Kol96}.

For nearly five decades, this bend-and-break framework was the general
route to rational curves on Fano manifolds. Du, Guo, and Xie~\cite{DGX26}
give an analytic construction under the sole Fano hypothesis:
positive Ricci curvature controls holomorphic disks, producing
finite-area entire curves that extend to rational curves through
prescribed points. Here we return to the curvewise hypothesis of
Miyaoka--Mori. Weighted jets replace Frobenius amplification, and the
resulting ruled surface produces rational curves through the points
of the given curve.

\begin{theorem}[Miyaoka--Mori]\label{thm:main}
Let $\Bbbk$ be an algebraically closed field of characteristic zero,
let $X$ be a smooth projective variety over $\Bbbk$, and let
$f:C\to X$ be a nonconstant morphism from a smooth connected
projective curve over $\Bbbk$. If
\[
 -K_X\cdot f_*[C]>0,
\]
then, for every closed point $x\in f(C)$, there is a nonconstant morphism
$b:\PP^1\to X$ with $b(0)=x$.
\end{theorem}

The hypothesis imposes positivity of $-K_X$ only on the given curve;
$-K_X$ need not be nef, and $f(C)$ need not move in a covering family.
The construction also yields an effective numerical decomposition of
$f_*[C]$. In the theorem below, $\equiv$ denotes numerical equivalence
of one-cycles, and an effective $\mathbb Q$-cycle means a finite sum of
integral curves with nonnegative rational coefficients; its components
need not be rational curves.

\begin{theorem}[Effective domination]\label{thm:effective-domination}
Under the hypotheses of Theorem~\ref{thm:main}, there is a rational
number $\mu>0$ with the following property. For every closed point
$x\in f(C)$, one can choose an integral rational curve $R_x\subset X$
through $x$ and an effective $\mathbb Q$-cycle $Q_x$ on $X$ such that
\[
 f_*[C]\equiv\mu R_x+Q_x,
 \qquad -K_X\cdot R_x\leq\dim X+1.
\]
In particular, this same curve $R_x$ satisfies
\[
 P\cdot R_x\leq\frac{P\cdot f_*[C]}{\mu}
\]
for every nef real Cartier divisor $P$ on $X$.
\end{theorem}

The coefficient $\mu$ is independent of $x$, and for each $x$ the curve
$R_x$ satisfies the displayed inequality for every nef divisor $P$.
The theorem does not assert that $-K_X\cdot R_x$ is positive.
We establish the effective decomposition before applying two-point
bend-and-break; this final step preserves $\mu$ and yields the stated
upper bound on $-K_X\cdot R_x$.

\subsection{The construction}
Set $n=\dim X$. Choose an embedding $X\hookrightarrow\PP^N$ and put
$A=f^*\OO_X(1)$. Write
\[
 a=\deg A>0,\qquad d=-K_X\cdot f_*[C]>0,\qquad
 h_k=\sum_{q=1}^k\frac1q.
\]
Let $F_1,\ldots,F_\nu$ be homogeneous equations defining
$X\subset\PP^N$. If $\nu>0$, set
$\delta=\max_{1\leq j\leq\nu}\deg F_j$; if $X=\PP^N$, take
$\nu=0$ and $\delta=1$.
Since $h_k\to\infty$, we can choose $k$ with
\[
 d h_k>2(n+1)\delta a.
\]

The homogeneous coordinate sections pull back to
$f_0,\ldots,f_N\in H^0(C,A)$. For each $c\in C$, the tuple
$s(c)=(f_0(c),\ldots,f_N(c))$ is a nonzero point of the affine cone over
$X$ in $A_c^{\oplus(N+1)}$. We consider germs of maps
$t\mapsto\gamma(t)$ into this cone based at $s(c)$, retaining their
expansions up to order $k$. Over $\mathbb C$ these are holomorphic
map germs. After a local trivialization of $A$, such a $k$-jet has the
ambient expansion
\[
 s(c)+B_1t+\cdots+B_kt^k\pmod{t^{k+1}},
\]
where every $F_j$ vanishes after substitution modulo $t^{k+1}$.
Because $X$ is smooth, the cone is smooth of dimension $n+1$ near
$s(c)$. Thus a based jet has $n+1$ local coefficients at each order
$q=1,\ldots,k$. The rescaling $t\mapsto\lambda t$ gives the order-$q$
coefficients weight $q$. Quotienting the nonconstant jets by this
action gives a family of weighted projective spaces
$\pi_k:Y_k^{\mathrm{GG}}\to C$. Each fiber has $n+1$ coordinates of
each weight $1,\ldots,k$, so the total dimension is $D:=(n+1)k$.

For $m>0$ sufficiently divisible, define the rational tautological
class
\[
 H_k:=\frac1m c_1\bigl(\OO_{Y_k^{\mathrm{GG}}}(m)\bigr).
\]
It is ample on each fiber. With
$v_k=H_k^{D-1}\cdot[\pi_k^{-1}(c)]=1/(k!)^{n+1}>0$, the intersection
formula gives
\[
 H_k^D=-v_k d h_k<0.
\]
Thus $H_k$ has negative top self-intersection on the total space,
although it has positive degree on every curve contained in a fiber.
Fix $p\in C$ and set $\theta=d h_k/(2D)$. Then
\[
 (H_k+\theta\pi_k^*[p])^D=-\tfrac12v_kd h_k<0.
\]
This class is not nef, so it has negative degree on some integral curve
$\Gamma\subset Y_k^{\mathrm{GG}}$. Since its degree on every vertical
curve is positive, $\Gamma$ maps onto $C$. From
$\pi_k^*[p]\cdot\Gamma=\deg(\Gamma\to C)$ we obtain
\[
 -\frac{H_k\cdot\Gamma}{\deg(\Gamma\to C)}
 >\frac{d h_k}{2D}.
\]

The inverse tautological class becomes an ordinary line bundle after a
finite cover of the normalization of $\Gamma$: write $\widetilde C$ for
the covering curve, and $\rho:\widetilde C\to C$ and
$\tau:\widetilde C\to Y_k^{\mathrm{GG}}$ for the induced maps. For
sufficiently divisible $m$, a line bundle $L$ on $\widetilde C$
satisfies
\[
 \tau^*\OO_{Y_k^{\mathrm{GG}}}(m)\simeq L^{-m}.
\]
On this cover, $L$ also carries an affine lift of the weighted jet.
Taking degrees gives the positive rational number
\[
 \mu:=\frac{\deg L}{\deg\rho}
 =-\frac{H_k\cdot\Gamma}{\deg(\Gamma\to C)}
 >\frac{d h_k}{2D}>0.
\]
This is the coefficient in Theorem~\ref{thm:effective-domination}, and
it is fixed before the prescribed point $x$ is chosen.

The positive degree of $L$ now forces the high-order coefficients to
vanish. For $1\leq q\leq k$, each order-$q$ ambient coefficient of the
lifted jet is a section of $\rho^*A\otimes L^{-q}$, a line bundle of
degree $\deg\rho\,(a-q\mu)$, so it vanishes once $q>a/\mu$. The chosen
bound on $k$, together with the lower bound on $\mu$, gives
$a/\mu<k/\delta$. Hence the surviving coefficients form a polynomial in
the fiber coordinate of $\Tot(L)$ of degree less than $k/\delta$.
Substitution into each $F_j$ produces a polynomial of
degree less than $k$. The jet condition says that this polynomial
vanishes modulo $t^{k+1}$, so it vanishes identically. The truncated
jet therefore extends to a polynomial map from $\Tot(L)$ to the
pullback of the twisted affine cone over $\widetilde C$.

The two arguments create a gap by amplifying different quantities.
Frobenius multiplies the anticanonical degree of a map from a curve
without changing its genus, providing enough deformations for
bend-and-break. Here increasing the jet order introduces the unbounded
factor $h_k$ into the normalized tautological intersection. The resulting
slope makes the degree of the surviving polynomial $o(k)$, so the defining
equations, which vanish through order $k$, must eventually hold
identically. Rational curves then arise from the ruled surface;
bend-and-break enters only afterward to bound their anticanonical degree.

The induced rational map to $X$ is nonconstant on a general fiber of
$\Tot(L)\to\widetilde C$. Suppose that its jet were constant after
projectivization; then every positive-order ambient coefficient would
be a scalar multiple of the base tuple $\rho^*s$, and those scalar
multiples would be sections of $L^{-q}$ for $q>0$. All of them vanish
because $\deg L>0$, contradicting the nonconstant based jet.
Compactify $\Tot(L)$ to the ruled surface
$\Pline(\OO_{\widetilde C}\oplus L)$. The rational map to $X$ can be
resolved by point blowups away from the zero section, so that the zero
section continues to map to $f\circ\rho$. The reduced support of a
fiber of the resolved surface is a connected tree of rational curves.
The total $\OO_X(1)$-degree on every fiber is positive, so each fiber
has a noncontracted component. Starting at the zero section and
following a chain to the first such component gives a rational curve
through any prescribed point of $f(C)$.
Comparing the zero and infinity sections yields the effective
decomposition in Theorem~\ref{thm:effective-domination}, with the same
$\mu$ for every prescribed point.

The harmonic sum $h_k$ also appears in Demailly's curvature
calculation for Green--Griffiths jets. For a directed bundle $(V,h)$
of rank $r$, Demailly's formula~\cite[(2.18)]{Dem11} for the averaged
horizontal curvature reads
\[
 \mathbb E(g_k)=\frac{h_k}{kr}\,\Theta_{\det(V^*),\det h^*}.
\]
When $V=T_M$ for a complex projective manifold $M$, the determinant on
the right is $K_M$. Merker's 2010 preprint constructed negatively
twisted jet differentials for smooth projective hypersurfaces of
general type~\cite{Mer15}, and Demailly then used the bigness of $K_M$
to obtain them on arbitrary projective manifolds of general
type~\cite[Theorem~0.5]{Dem11}. By the fundamental vanishing theorem of
entire curves~\cite{GG80,SY97,Dem11}, these differentials vanish on
the jets of every entire curve, and their differential equations can
then constrain, and in favorable cases exclude, entire curves.

Our based jets live in a twisted affine cone along $C$, and $\Gamma$
is a curve in their weighted projectivization rather than a jet
differential. The reciprocal jet weights nevertheless give the same
$h_k$, while the canonical contribution here is
$h_k(K_X\cdot f_*[C])=-d h_k<0$. The contrast is striking: in
hyperbolicity theory the jet calculation supplies obstructions to
entire curves, whereas here the negative tautological degree of $\Gamma$
produces a positive-degree parameter line, whose jets extend to
nonconstant maps $\mathbb A^1\to X$ on general fibers and then to
rational curves $\PP^1\to X$. The same harmonic weight thus unites
the two uses of jets: constraining curves and constructing them.

\subsection{Related work and applications}
The coefficient-vanishing step has a parallel in Campana--P\u{a}un's
algebraicity argument for foliations of positive minimal
slope~\cite[\S4.1, Proposition~4.5 and Lemma~4.6]{CamP19}.
There, positivity forces sufficiently high Taylor terms in the
foliation directions to vanish, bounding the growth of sections on
the Zariski closure of the leaf relation. Here $\deg L>0$ makes the
high-order ambient coefficients vanish, and the degrees of the defining
equations turn the remaining finite jet into a polynomial map.

For a projective variety over an algebraically closed field of arbitrary
characteristic, Jovinelly, Lehmann, and Riedl~\cite[Theorem~1.1]{JLR26}
sharpen the Miyaoka--Mori bound. Let $C$ be an integral curve in the
smooth locus of $X$ with $K_X\cdot C<0$.
For every nef real Cartier divisor $H$ and closed point $x\in C$,
they find a rational curve $R$ through $x$ such that
\[
 H\cdot R\leq(\dim X+1)\frac{H\cdot C}{-K_X\cdot C}.
\]
The constant $\dim X+1$ is optimal. We do not claim an optimal bound
for $\mu$; our choice of jet order depends on the embedding and its
defining equations. The effective decomposition fixes $\mu$
independently of $x$ and, for each $x$, gives one curve $R_x$ whose
inequality in Theorem~\ref{thm:effective-domination} holds for every
nef divisor.

We use BDPP cone duality~\cite[Theorem~0.2]{BDPP13} as an external
input. If $K_X$ is not pseudoeffective, it gives a covering family of
curves of negative canonical degree. The embedding-degree bound from
our construction is uniform along this family, so the resulting
rational curves form a covering family. In this way we recover
the projective uniruledness criterion in characteristic
zero~\cite[Corollary~0.3 and Theorem~2.6]{BDPP13}.

\begin{theorem}[BDPP]\label{thm:bdpp}
Let $X$ be a smooth connected complex projective variety of positive
dimension. Then $X$ is uniruled if and only if $K_X$ is not
pseudoeffective.
\end{theorem}

Ou proved the compact K\"ahler analogue~\cite[Theorem~1.1]{Ou25}, and
Cao--P\u{a}un gave a further proof and
refinement~\cite[Corollaries~5.3 and~6.3]{CP25}. Tosatti subsequently
established transcendental cone duality~\cite[Theorem~1.1 and
Corollary~1.2]{Tos26}.

Section~\ref{sec:weighted-jets} constructs the weighted jet space,
computes its tautological intersection, and produces a negative
horizontal curve. Section~\ref{sec:parameter-line} realizes the
inverse tautological class by a line bundle carrying an affine jet.
Sections~\ref{sec:realization} and~\ref{sec:completion} turn that jet
into rational curves through prescribed points, and
Sections~\ref{sec:numerical} and~\ref{sec:bdpp} give the effective
decomposition and the projective uniruledness criterion.

\section{Weighted jets and the intersection formula}
\label{sec:weighted-jets}

Retain the data $f:C\to X$ of Theorem~\ref{thm:main}, over an
algebraically closed field $\Bbbk$ of characteristic zero. All schemes
and morphisms in the construction are over $\Bbbk$. Points of curves
used to specify values of morphisms from $\PP^1$ are closed points.
Fix an embedding $X\hookrightarrow\PP^N$ with homogeneous coordinates
$[z_0:\cdots:z_N]$. Put $A=f^*\OO_X(1)$, and write
\[
 f_i=f^*(z_i|_X)\in H^0(C,A),\qquad 0\leq i\leq N.
\]
These coordinate sections have no common zero. The construction uses the
numerical quantities
\[
 \begin{aligned}
 n&=\dim X\geq1,\qquad a=\deg A>0,\\
 d&=-K_X\cdot f_*[C]=\deg f^*T_X>0.
 \end{aligned}
\]

Choose homogeneous generators $F_1,\ldots,F_\nu\in\Bbbk[z_0,\ldots,z_N]$
of the homogeneous ideal $I_X$ of $X\subset\PP^N$. Write $d_j=\deg F_j$
and choose $\delta\geq1$ with $d_j\leq\delta$ for every $j$.
If $I_X=0$, take $\delta=1$.

\subsection{The twisted affine cone}
For a vector bundle $\mathcal F$ on $C$, write
$\Tot(\mathcal F)=\Spec_C\operatorname{Sym}(\mathcal F^\vee)$ for its total space.
For a morphism $u:B\to C$, a map $B\to\Tot(\mathcal F)$ over $C$ is the same as
a section of $u^*\mathcal F$. In particular, the fiber over $c\in C$ is the
vector space $\mathcal F_c$.
Twisting the cone by $A=f^*\OO_X(1)$ makes the homogeneous coordinate
sections into a nowhere-zero section of the cone. The additional scalar
tangent direction has degree zero, so the relative tangent bundle along
this section has the same degree as $f^*T_X$.

\begin{definition}\label{def:cone}
The \emph{twisted affine cone} is the closed subscheme
\[
 \mathcal Z\subset\Tot\bigl(A^{\oplus(N+1)}\bigr)
\]
cut out by $F_j=0$, where each $F_j$ is evaluated on the fiber
coordinates and is a section of the pullback of $A^{\otimes d_j}$
from $C$. Write $\mathcal Z^\times$ for the complement of the zero
section.
\end{definition}

In a local frame $e$ of $A$, a vector has the form
$(u_0e,\ldots,u_Ne)$. If $e'=he$, its coordinates become
$u_i'=h^{-1}u_i$, and homogeneity gives
$F_j(u')=h^{-d_j}F_j(u)$. Thus the local zero loci glue.

Each $f_i(c)$ lies in $A_c$. The sections $f_i$ satisfy every
$F_j$ and have no common zero, so their tuple gives a section of the
natural projection $\mathcal Z^\times\to C$:
\[
 s:C\longrightarrow\mathcal Z^\times,\quad
 c\longmapsto\bigl(c;f_0(c),\ldots,f_N(c)\bigr).
\]

Projectivization gives a morphism $p:\mathcal Z^\times\to C\times X$.
In a local frame $e$ of $A$, it is
\[
 \bigl(c;u_0e(c),\ldots,u_Ne(c)\bigr)
 \longmapsto\bigl(c,[u_0:\cdots:u_N]\bigr).
\]
Take the line bundle
\[
 \mathcal M=\operatorname{pr}_C^*A\otimes
 \operatorname{pr}_X^*\OO_X(-1)
\]
on $C\times X$. The inclusion
$\OO_X(-1)\hookrightarrow\OO_X^{\oplus(N+1)}$ gives a morphism from
the punctured total space $\Tot_{C\times X}(\mathcal M)^\times$ to
$\Tot_C(A^{\oplus(N+1)})$: a nonzero tensor over $(c,x)$ is sent to the
corresponding nonzero vector in $A_c^{\oplus(N+1)}$. Its projective
class is $x$, so the defining equations $F_j$ vanish. Conversely, a
nonzero vector of $\mathcal Z_c$ determines its projective point
$x\in X$ and is then a nonzero element of
$A_c\otimes\OO_X(-1)_x$. The two constructions are inverse in local
trivializations and hence give an isomorphism over $C\times X$
\begin{equation}\label{eq:punctured-line-bundle}
 \mathcal Z^\times\simeq
 \Tot_{C\times X}(\mathcal M)^\times.
\end{equation}
In particular, $\mathcal Z^\times$ is smooth over $C$ of relative
dimension $n+1$. Under \eqref{eq:punctured-line-bundle}, the composite
$p\circ s$ equals $(\mathrm{id}_C,f)$: the section $s$ is a lift of the
graph of $f$.
Set
\[
 E=s^*T_{\mathcal Z^\times/C},\qquad
 E_c=T_{\mathcal Z_c^\times,s(c)}.
\]
Thus $E$ records the tangent directions within the cone fibers along $s$.
For the punctured line bundle in \eqref{eq:punctured-line-bundle}, the
relative tangent sequence over $C$ has vertical scalar direction
$\OO$ and quotient the pullback of $T_X$. Pulling this sequence back by
$s$, whose image in $C\times X$ is the graph of $f$, gives the relative
tangent sequence
\begin{equation}\label{eq:cone-tangent}
 0\longrightarrow\OO_C\longrightarrow E\longrightarrow f^*T_X\longrightarrow0.
\end{equation}
The first term is trivialized by the infinitesimal scalar action at
$s$. In particular,
\begin{equation}\label{eq:cone-degree}
 \operatorname{rk}E=n+1,\qquad \deg E=d.
\end{equation}
The cone therefore adds one scalar tangent direction and leaves the
degree $d$ unchanged.

\subsection{Relative based jets and weighted projectivization}
Fix a jet order \(k\geq1\), and let \(V\) be a smooth variety of dimension
\(\ell\). The \(k\)-jet scheme \(J_k(V)\) parametrizes infinitesimal curves
\[
 \gamma:\Spec\bigl(\Bbbk[t]/(t^{k+1})\bigr)\longrightarrow V.
\]
For \(v\in V\), write
\[
 J_k(V)_v:=\{\gamma\in J_k(V):\gamma(0)=v\}
\]
for the space of \(k\)-jets based at \(v\). In centered local étale coordinates
\(x_1,\ldots,x_\ell\) at \(v\), every such jet has a unique expansion
\[
 x_i(t)=x_{i,1}t+\cdots+x_{i,k}t^k,
 \qquad 1\leq i\leq \ell.
\]
Reparametrization \(t\mapsto\lambda t\) sends
\[
 x_{i,q}\longmapsto \lambda^q x_{i,q},
 \qquad \lambda\in\Bbbk^\times,
\]
so the coefficient of order \(q\) has weight \(q\).

We return to the smooth morphism $\mathcal Z^\times\to C$ and its section
$s$. A relative based $k$-jet over a point $c\in C$ is such an
infinitesimal curve in $\mathcal Z_c^\times$, based at $s(c)$.
Scheme-theoretically, the projection
$\mathcal Z^\times\to C$ induces
\[
 J_k(\mathcal Z^\times)\longrightarrow J_k(C),
\]
and the constant jets give a natural section $C\to J_k(C)$. Thus
\[
 J_k(\mathcal Z^\times/C)
 :=J_k(\mathcal Z^\times)\times_{J_k(C)}C
\]
selects the jets whose projection to $C$ is constant in the jet
parameter.
Evaluation at $t=0$ sends a jet to its center. To require a relative jet
$\gamma\in J_k(\mathcal Z^\times/C)$ lying over $c\in C$ to satisfy
$\gamma(0)=s(c)$, we further take the fiber product with the section
$s:C\to\mathcal Z^\times$:
\[
 J_k^s=J_k^s(\mathcal Z/C)
 :=J_k(\mathcal Z^\times/C)\times_{\mathcal Z^\times,s}C.
\]
Thus, over a point \(c\in C\), a relative based jet has an expansion
in the ambient vector bundle
\begin{equation}\label{eq:ambient-jet}
 s(c)+B_1t+\cdots+B_kt^k,
 \qquad
 B_q\in A_c^{\oplus(N+1)},
\end{equation}
where substitution into the defining equations \(F_j\) of the cone
vanishes modulo \(t^{k+1}\).

The vectors \(B_q\) give an ambient description, but the equations
\(F_j\) constrain them. Since
$\mathcal Z^\times\to C$ is smooth of relative dimension $n+1$ and has
the section $s$, after refining an affine open cover
$\{U_\alpha\}$ of $C$ we may choose, for each $\alpha$, an open
neighborhood $V_\alpha$ of $s(U_\alpha)$ and an étale morphism over
$U_\alpha$
\[
 (\operatorname{pr}_C\circ p,z_\alpha):V_\alpha\longrightarrow
 U_\alpha\times\mathbb A^{n+1}
\]
which sends $s|_{U_\alpha}$ to the zero section. We may refine further
so that $E|_{U_\alpha}$ is trivial. Writing
$z_{\alpha,1},\ldots,z_{\alpha,n+1}$ for the affine coordinates, the
étale base-change property of jet schemes~\cite[Lemma~2.9]{EM09} and
the two base changes that impose constant projection to $C$ and center
$s$ together give
\begin{equation}\label{eq:based-jet-chart}
 J_k^s|_{U_\alpha}
 \simeq U_\alpha\times\mathbb A^{(n+1)k}.
\end{equation}
Write \(x_{\alpha,i,q}\) for the coefficient of \(t^q\) in
\(z_{\alpha,i}\), and put
\[
 x_{\alpha,q}=(x_{\alpha,i,q})_{i=1}^{n+1},
 \qquad
 x_\alpha=(x_{\alpha,1},\ldots,x_{\alpha,k}).
\]
For a fixed jet, its $x_{\alpha,i,q}$ and $B_q$ describe the same
truncated map in centered étale and ambient cone coordinates,
respectively. The change between them is generally nonlinear.

Let $q_k:J_k^s\to C$ be the structure morphism and set
\[
 \mathcal S_k=(q_k)_*\OO_{J_k^s}.
\]
Reparametrization $t\mapsto\lambda t$ induces a grading of this sheaf of
algebras in which $x_{\alpha,i,q}$ has weight $q$. The jet-coordinate
changes commute with this action, hence preserve the grading. Locally,
\[
 \mathcal S_k|_{U_\alpha}
 \simeq
 \OO_{U_\alpha}
 [x_{\alpha,i,q}:1\leq i\leq n+1,\ 1\leq q\leq k],
 \qquad \deg x_{\alpha,i,q}=q,
\]
and globally
\[
 \mathcal S_k=\bigoplus_{m\geq0}(\mathcal S_k)_m,
 \qquad (\mathcal S_k)_0=\OO_C.
\]
Each $(\mathcal S_k)_m$ is locally free, with local basis the monomials of
weight $m$ in the coordinates $x_{\alpha,i,q}$.

We write \(\Pline(\mathcal F)\) for the bundle of lines in a vector
bundle \(\mathcal F\), with tautological subline
$\OO_{\Pline(\mathcal F)}(-1)$.
Define the relative weighted projectivization by
\[
 Y_k^{\mathrm{GG}}:=\Proj_C\mathcal S_k,
 \qquad \pi_k:Y_k^{\mathrm{GG}}\longrightarrow C.
\]
The scheme $Y_k^{\mathrm{GG}}$ is the geometric quotient of the
nonconstant based jets by the scalar action $t\mapsto\lambda t$;
we do not quotient by the full group of invertible changes of the jet
parameter.
The superscript $\mathrm{GG}$ refers to the Green--Griffiths construction;
see \cite[(0.3) and the discussion preceding (0.13)]{Dem11}. Locally,
\[
 Y_k^{\mathrm{GG}}|_{U_\alpha}
 \simeq U_\alpha\times
 \PP(1^{n+1},2^{n+1},\ldots,k^{n+1}),
\]
where the notation on the right denotes the weighted projective space
with $n+1$ homogeneous coordinates of each weight $1,\ldots,k$.
For \(k=1\), the based jets identify with \(E\), so
\(
 Y_1^{\mathrm{GG}}=\Pline(E).
\)
For higher orders the transition maps are generally nonlinear; the
direct sum of copies of $E$ will arise as a deformation below.

\subsection{The tautological intersection}
For the weighted jet formalism, see \cite{GG80} and
\cite[\S6]{Dem97}. Our goal is the formula of
Proposition~\ref{prop:harmonic}, whose harmonic factor is
$\sum_{q=1}^k1/q$. We first record the tautological polarization, then
deform the jet transitions to their linear parts, and finally split the
underlying vector bundle, so that coordinate hyperplanes become global.

\begin{lemma}[The Veronese polarization]\label{lem:veronese-polarization}
For every $k\geq1$, the scheme $Y_k^{\mathrm{GG}}$ is normal and projective, of dimension
$(n+1)k$. For $m>0$ sufficiently divisible,
$\OO_{Y_k^{\mathrm{GG}}}(m)$ is invertible and relatively very ample.
The rational tautological class
\[
 H_k=\frac{1}{m}c_1\bigl(\OO_{Y_k^{\mathrm{GG}}}(m)\bigr)
\]
is independent of the sufficiently divisible choice of $m$; we use it
as a rational Cartier class in intersection products. Its degree on a
fiber is
\begin{equation}\label{eq:fiber-degree}
 v_k=H_k^{(n+1)k-1}\cdot[\pi_k^{-1}(c)]
 =\frac{1}{(k!)^{n+1}},
\end{equation}
independently of $c\in C$.
\end{lemma}

\begin{proof}
Since $\mathcal S_k$ is locally a finitely generated weighted polynomial
algebra, a sufficiently divisible Veronese
$\mathcal S_k^{(m)}=\bigoplus_{\ell\geq0}(\mathcal S_k)_{\ell m}$ is generated
in degree one~\cite[Tag~0EGH]{Stacks}. The local algebras all have the
same generator weights, so one such $m$, divisible by every $q\leq k$,
serves for all of them; it also serves for the deformation families
below, whose local graded algebras have the same weights. The degree-one
part $(\mathcal S_k)_m$ is locally free, and multiplication gives
\[
 \operatorname{Sym}_{\OO_C}((\mathcal S_k)_m)\twoheadrightarrow\mathcal S_k^{(m)}.
\]
A Veronese does not change relative Proj. Hence
\[
 Y_k^{\mathrm{GG}}=\Proj_C\mathcal S_k^{(m)}\hookrightarrow
 \Proj_C\operatorname{Sym}_{\OO_C}((\mathcal S_k)_m)
 =\Pline((\mathcal S_k)_m^\vee).
\]
The dual reflects our convention of projectivizing lines. The
hyperplane bundle $\OO_{\Pline((\mathcal S_k)_m^\vee)}(1)$ restricts to
$\OO_{Y_k^{\mathrm{GG}}}(m)$, which is therefore invertible and
relatively very ample. Further divisible multiples give the same
rational class $H_k$.
By the local product description
\[
 Y_k^{\mathrm{GG}}|_U\simeq
 U\times\PP(1^{n+1},2^{n+1},\ldots,k^{n+1})
\]
and the standard properties of weighted projective space
\cite[Proposition~1.3.3(i)]{Dol82}, the scheme $Y_k^{\mathrm{GG}}$ is
integral and normal, with fiber dimension $(n+1)k-1$ and total dimension
$(n+1)k$. Since $C$ and the morphism $Y_k^{\mathrm{GG}}\to C$ are
projective, $Y_k^{\mathrm{GG}}$ is projective over $\Bbbk$.

It remains to compute the fiber degree. On a fiber over $c\in C$,
consider the finite morphism
\[
 \varphi_k:\PP^{(n+1)k-1}\longrightarrow
 \PP(1^{n+1},2^{n+1},\ldots,k^{n+1}),
 \qquad [u_{i,q}]\longmapsto[u_{i,q}^{\,q}].
\]
The graded ring map is integral, so $\varphi_k$ is finite. On a chart
where a weight-$1$ coordinate is normalized to $1$ in source and target,
the remaining coordinates map as $u_{i,q}\mapsto u_{i,q}^{\,q}$. A general
target point has $q$ choices for each weight-$q$ coordinate, independently;
hence $\deg\varphi_k=\prod_{q=1}^k q^{n+1}=(k!)^{n+1}$.
For $m$ divisible by every $q\leq k$,
\[
 \varphi_k^*\OO(m)\simeq\OO_{\PP^{(n+1)k-1}}(m),
\]
so $\varphi_k^*H_k$ is the ordinary hyperplane class. The projection
formula therefore gives
\[
 (\deg\varphi_k)v_k=1,
\]
which is \eqref{eq:fiber-degree}.
\end{proof}

To compute $H_k^{(n+1)k}$, we use the fact that the coordinate
hyperplanes have empty common intersection. These hyperplanes are only
locally defined: nonlinear jet-coordinate changes, and even their linear
parts, generally mix the coordinate functions. We therefore deform the
jet-coordinate transitions to their linear parts, and then deform $E$ to
a direct sum of line bundles; in the latter model the individual
coordinate hyperplanes become global divisors. Write $Y^{\mathrm{sp}}$
for the resulting split weighted projective bundle, with weights
$1,\ldots,k$. The next lemma shows that these two deformations preserve
the required intersection numbers.

\begin{lemma}[Reduction to a split weighted bundle]\label{lem:split-reduction}
There are line bundles $Q_1,\ldots,Q_{n+1}$ on $C$ with
$\sum_i\deg Q_i=d$ such that the following holds. Set
\[
 Y^{\mathrm{sp}}=\Proj_C\operatorname{Sym}_{\OO_C}
 \left(\bigoplus_{q=1}^k\bigoplus_{i=1}^{n+1}Q_i^\vee\right),
\]
where the generators in the $q$-th block have degree $q$, and let
$\pi_{\mathrm{sp}}:Y^{\mathrm{sp}}\to C$ be the structure morphism.
For $m$ sufficiently divisible, put
\[
 H^{\mathrm{sp}}=\frac{1}{m}c_1\bigl(\OO_{Y^{\mathrm{sp}}}(m)\bigr),
\]
using the twisting sheaf of this graded Proj. Then
\[
 (H^{\mathrm{sp}})^{(n+1)k}=H_k^{(n+1)k},
 \qquad
 (H^{\mathrm{sp}})^{(n+1)k-1}\cdot[\pi_{\mathrm{sp}}^{-1}(c)]=v_k
 \quad(c\in C).
\]
\end{lemma}

\begin{proof}
\proofheading{Removing the nonlinear terms.}
Recall that $Y_k^{\mathrm{GG}}=\Proj_C\mathcal S_k$ is glued by the
graded transition maps of the relative based-jet coordinates $x_\alpha$.
Each relative étale coordinate map identifies the based section with
the zero section over $U_\alpha$. Formal étaleness therefore identifies
their formal neighborhoods, with completed structure sheaf
$\OO_{U_\alpha}[[z_{\alpha,1},\ldots,z_{\alpha,n+1}]]$.
On $U_\alpha\cap U_{\alpha'}$, composing these identifications gives
a formal coordinate change
\[
 z_{\alpha'}=\Phi_{\alpha\alpha'}(z_\alpha),
 \qquad \Phi_{\alpha\alpha'}(0)=0,
\]
with coefficients in $\OO_C(U_\alpha\cap U_{\alpha'})$. We truncate
this formal change to the finite order relevant modulo $t^{k+1}$; hence
only finitely many terms occur, all with coefficients regular on the base
overlap. Substituting
$z_\alpha(t)=\sum_{q=1}^k x_{\alpha,q}t^q$ therefore gives transition
maps that are polynomial and weighted:
\begin{equation}\label{eq:jet-transition}
 \begin{aligned}
 x_{\alpha'}&=G_{\alpha\alpha'}(x_\alpha),\\
 x_{\alpha',q}&=g_{\alpha\alpha'}x_{\alpha,q}
  +P_{\alpha\alpha',q}(x_{\alpha,1},\ldots,x_{\alpha,q-1}).
 \end{aligned}
\end{equation}
Here $g_{\alpha\alpha'}$ is the coordinate transition matrix of $E$:
column coordinates of a tangent vector satisfy
$v_{\alpha'}=g_{\alpha\alpha'}v_\alpha$. The same matrix occurs in every
jet order $q$. The polynomial
$P_{\alpha\alpha',q}$ has regular coefficients on the overlap, and
each of its monomials has weight $q$ and ordinary degree at least two.
Only finitely many Taylor terms can contribute modulo $t^{k+1}$, and
the inverse coordinate change has the same form.

The triangular form of \eqref{eq:jet-transition} underlies the standard
filtration of Green--Griffiths jet-differential bundles and its
associated graded decomposition into tensor products of symmetric
powers; see \cite[\S6, (6.5)]{Dem97} and
\cite[\S2.3, Theorem~2.3]{Mer15}. For our relative based-jet space we
realize the analogous passage by deforming these transitions to their
linear parts while keeping the jet weights fixed.
Introduce a parameter $\lambda$ and set
\[
 G_{\alpha\alpha'}^{(\lambda)}(x)
   =\lambda^{-1}G_{\alpha\alpha'}(\lambda x),
 \qquad x=(x_1,\ldots,x_k).
\]
Here $\lambda$ multiplies every coordinate by the same factor,
independently of jet weight, so a nonlinear monomial of ordinary degree
$e\geq2$ acquires the factor $\lambda^{e-1}$ while retaining its jet
weight. Since there is no constant term, both the transition
maps and their inverses are polynomial in $\lambda$, and at
$\lambda=0$ only the linear terms remain. Their cocycle identities hold
for invertible $\lambda$ and hence, as polynomial identities, also at
$\lambda=0$; the inverse identities extend for the same reason, so the
specialized maps at zero still define mutually inverse coordinate
changes. Every transition preserves the original jet grading.

Consequently these transitions glue the varieties
$(U_\alpha\times\mathbb A^1)\times
 \PP(1^{n+1},\ldots,k^{n+1})$
into a flat family $\mathcal Y\to C\times\mathbb A^1$ that is locally a
product over $C\times\mathbb A^1$. Let $\mathcal Y_\lambda$ denote
the fiber of the composite $\mathcal Y\to\mathbb A^1$ over $\lambda$.
Then $\mathcal Y_1=Y_k^{\mathrm{GG}}$, while
$\mathcal Y_0$ is the weighted projectivization of copies of $E$ in
weights $1,\ldots,k$, with transitions
$x_{\alpha',q}=g_{\alpha\alpha'}x_{\alpha,q}$.

With the common $m$ chosen in Lemma~\ref{lem:veronese-polarization},
the Veronese construction gives a line bundle $\OO_{\mathcal Y}(m)$
and a relative projective embedding of $\mathcal Y$ over
$C\times\mathbb A^1$. Hence $\mathcal Y$ is projective over
$\mathbb A^1$, and
\[
 \OO_{\mathcal Y}(m)|_{\mathcal Y_1}
 \simeq\OO_{Y_k^{\mathrm{GG}}}(m).
\]

Since $\mathcal Y\to\mathbb A^1$ is flat and projective and
$\OO_{\mathcal Y}(m)$ is invertible, every tensor power
$\OO_{\mathcal Y}(m)^{\otimes\ell}$ is flat over $\mathbb A^1$. Hence,
for every $\ell\geq0$,
\[
 \chi\!\left(\mathcal Y_\lambda,
 \bigl(\OO_{\mathcal Y}(m)|_{\mathcal Y_\lambda}\bigr)^{\otimes\ell}\right)
\]
is locally constant in $\lambda$~\cite[Tag~0B9T]{Stacks}, hence constant
since $\mathbb A^1$ is connected. Multiplying the coefficient of
$\ell^{(n+1)k}$ in this Euler-characteristic polynomial by
$((n+1)k)!$ gives
$c_1(\OO_{\mathcal Y}(m)|_{\mathcal Y_\lambda})^{(n+1)k}$~\cite[Theorem~1.1.24]{Laz04}.
Dividing by $m^{(n+1)k}$ shows that the rational tautological top
intersection is constant. No ampleness of
$\OO_{\mathcal Y}(m)|_{\mathcal Y_\lambda}$ is required.

\proofheading{Splitting the vector bundle.}
We start from $\mathcal Y_0$. To make the individual coordinate
hyperplanes glue globally, we deform the underlying bundle $E$ to a
sum of line bundles and apply this deformation in each weight block.
Choose a filtration by subbundles
\[
 0=E_0\subset E_1\subset\cdots\subset E_{n+1}=E
\]
with line quotients $Q_i=E_i/E_{i-1}$. Such a filtration is obtained
by saturating a rational line in $E$ and repeating in the quotient:
a torsion-free sheaf on a smooth curve is locally free.

After further refining the open cover, choose on each $U_\alpha$ a frame
$e_{\alpha,1},\ldots,e_{\alpha,n+1}$ adapted to the filtration, with
$E_i$ spanned by the first $i$ frame vectors. In these frames the
coordinate transition matrices $\widetilde g_{\alpha\alpha'}$, defined
by $v_{\alpha'}=\widetilde g_{\alpha\alpha'}v_\alpha$, are upper triangular.
Put
\[
 \Lambda(\lambda)
 =\operatorname{diag}(\lambda,\lambda^2,\ldots,\lambda^{n+1}).
\]
Since the matrices are upper triangular, an entry can be nonzero only
for $i_1\leq i_2$. The corresponding entry of
$\Lambda(\lambda)^{-1}\widetilde g_{\alpha\alpha'}\Lambda(\lambda)$
carries the factor $\lambda^{i_2-i_1}$, whose exponent is nonnegative.
The same holds for the inverse upper-triangular matrices. Hence both
matrices extend to $\lambda=0$, where they become diagonal, and the
cocycle and inverse identities persist. This gives a vector-bundle
deformation from $E$ to $\bigoplus_i Q_i$. Taking copies in the
prescribed weights and using the same Veronese degree $m$ gives a second
projective flat family, whose top intersection is the same.

The final fiber is $Y^{\mathrm{sp}}$, and the two deformations give the
asserted equality of top intersections. The filtration gives
$\sum_i\deg Q_i=\deg E=d$. Over each fixed point $c\in C$, both
deformations retain the same weighted projective space, with the same
divisible tautological polarization. Thus the fiber degree of
$H^{\mathrm{sp}}$ is $v_k$.
\end{proof}

The split model makes the reciprocal jet weights visible. We abbreviate
\begin{equation}\label{eq:jet-constants}
 h_k=\sum_{q=1}^k\frac1q.
\end{equation}

\begin{proposition}[The weighted intersection]\label{prop:harmonic}
For every $k\geq1$, the tautological class $H_k$ satisfies
\begin{equation}\label{eq:harmonic-intersection}
 H_k^{(n+1)k}=-v_k d h_k
 =-\frac{d}{(k!)^{n+1}}\sum_{q=1}^k\frac1q.
\end{equation}
\end{proposition}

\begin{proof}
Apply Lemma~\ref{lem:split-reduction}, choosing $m$ sufficiently
divisible and divisible by every $q\leq k$. It remains to compute the
top intersection of $H^{\mathrm{sp}}$.
Choose local frames $e_i$ of the $Q_i$, and write $x_{i,q}$ for the fiber
coordinate in the $Q_i$ direction of the weight-$q$ block. We raise
$x_{i,q}$ to the $m/q$-th power, so that its weight becomes $m$, for which
$\OO_{Y^{\mathrm{sp}}}(m)$ is invertible; this does not change the
coordinate zero set.
If $e_i'=h_i e_i$, then $x_{i,q}'=h_i^{-1}x_{i,q}$. Thus the local
expressions $x_{i,q}^{m/q}\otimes\pi_{\mathrm{sp}}^*e_i^{\otimes m/q}$
glue to a global section
\[
 \sigma_{i,q}\in H^0\!\left(Y^{\mathrm{sp}},
 \OO_{Y^{\mathrm{sp}}}(m)\otimes \pi_{\mathrm{sp}}^*Q_i^{\otimes m/q}\right).
\]
These sections form a nowhere-zero section of the rank-$(n+1)k$
bundle
\[
 \mathcal V=\bigoplus_{i=1}^{n+1}\bigoplus_{q=1}^k
 \left(\OO_{Y^{\mathrm{sp}}}(m)\otimes
       \pi_{\mathrm{sp}}^*Q_i^{\otimes m/q}\right),
\]
since the homogeneous coordinates never vanish simultaneously.
Thus $\OO_{Y^{\mathrm{sp}}}\hookrightarrow\mathcal V$ has locally free
quotient of rank $(n+1)k-1$, and
$c_{(n+1)k}(\mathcal V)=0$~\cite[Tag~02UG]{Stacks}.
Applying this top Chern class to $[Y^{\mathrm{sp}}]$, taking its degree,
and dividing by $m^{(n+1)k}$ gives
\begin{equation}\label{eq:weighted-relation}
 \int_{Y^{\mathrm{sp}}}\prod_{i=1}^{n+1}\prod_{q=1}^k
 \left(H^{\mathrm{sp}}+\frac1q \pi_{\mathrm{sp}}^*c_1(Q_i)\right)=0.
\end{equation}
Products of two divisor classes pulled back from $C$ vanish. Expanding
\eqref{eq:weighted-relation} therefore gives
\[
 (H^{\mathrm{sp}})^{(n+1)k}
 =-v_k\sum_{i=1}^{n+1}\sum_{q=1}^k\frac{\deg Q_i}{q}
 =-v_k d h_k,
\]
where we used $\sum_i\deg Q_i=d$. The equality of top intersections
in Lemma~\ref{lem:split-reduction} proves \eqref{eq:harmonic-intersection}.
\end{proof}

The same harmonic factor appears in Demailly's curvature
average~\cite[(2.18)]{Dem11}; that analytic formula is not an input to
the intersection calculation above.

For $k=1$, Proposition~\ref{prop:harmonic} reads
$H_1^{n+1}=-\deg E$, since $v_1=h_1=1$.
The negative sign agrees with our convention of projectivizing lines.

\subsection{A negative horizontal curve}
The negative top intersection rules out nefness. The same argument,
applied to a pullback perturbation, gives a horizontal curve whose
negative degree is controlled relative to its degree over $C$.
Equivalently, the inverse tautological $\mathbb Q$-line-bundle class has
positive degree on the normalization of that curve.

\begin{lemma}[A negative horizontal curve]\label{lem:negative-horizontal}
Fix $k\geq1$ and put
\[
 \theta=\frac{d h_k}{2(n+1)k}.
\]
There exist an integral horizontal curve
$\Gamma\subset Y_k^{\mathrm{GG}}$, a smooth connected projective curve
$\widetilde C_0$, and a morphism
\[
 \tau_0:\widetilde C_0\longrightarrow
 \Gamma\subset Y_k^{\mathrm{GG}}
\]
which is the normalization of $\Gamma$ followed by the inclusion. If
$\rho_0=\pi_k\circ\tau_0$, then $\rho_0:\widetilde C_0\to C$ is finite
and surjective, and
\begin{equation}\label{eq:negative-intersection}
 H_k\cdot\Gamma+\theta\deg\rho_0<0,
 \qquad\text{equivalently}\qquad
 -\frac{H_k\cdot\Gamma}{\deg\rho_0}>\theta.
\end{equation}
\end{lemma}

\begin{proof}
The factor $1/2$ in $\theta$ is chosen for convenience.
Choose a closed point $p_0\in C$, and set
$H'_k=H_k+\theta\pi_k^*c_1(\OO_C(p_0))$.
Proposition~\ref{prop:harmonic} gives
\[
 (H'_k)^{(n+1)k}
 =v_k(-dh_k+(n+1)k\theta)
 =-\tfrac12v_kdh_k<0.
\]
A nef rational Cartier class on a projective variety has nonnegative
top self-intersection~\cite[Theorem~1.4.9]{Laz04}; no smoothness is
needed. Thus $H'_k$ is not nef, so $H'_k\cdot\Gamma<0$ for some integral
curve $\Gamma$. This curve is horizontal: on a vertical curve the class
pulled back from $C$ has degree zero, whereas the relatively ample class
$H_k$ has positive degree.

Let $\tau_0:\widetilde C_0\to\Gamma\subset Y_k^{\mathrm{GG}}$ be the
normalization followed by the inclusion, and put
$\rho_0=\pi_k\circ\tau_0$. Since $\Gamma$ is horizontal, $\rho_0$ is
finite and surjective. Since the closed point $p_0$ has degree one,
\[
 \pi_k^*c_1(\OO_C(p_0))\cdot\Gamma
 =\deg\rho_0^*\OO_C(p_0)=\deg\rho_0.
\]
Hence
$H'_k\cdot\Gamma=H_k\cdot\Gamma+\theta\deg\rho_0<0$, which is
\eqref{eq:negative-intersection}.
\end{proof}

\section{Realizing the inverse tautological class}\label{sec:parameter-line}

The fibers of $Y_k^{\mathrm{GG}}$ are weighted projective spaces, and they
may have cyclic quotient
singularities~\cite[Proposition~1.3.3(ii)]{Dol82}; the preceding
intersection argument needs only the rational Cartier class $H_k$ and
does not require these fibers to be smooth. The construction below,
however, works over smooth curves.

Let $\Gamma$, $\tau_0:\widetilde C_0\to Y_k^{\mathrm{GG}}$, and
$\rho_0:\widetilde C_0\to C$ be as in
Lemma~\ref{lem:negative-horizontal}; here $\widetilde C_0$ is the
smooth normalization of $\Gamma$. Our aim is to realize the pullback of
the inverse tautological $\mathbb Q$-line-bundle class by an ordinary
line bundle after finite base change, retaining the line-bundle class
rather than only its degree.
For $m$ sufficiently divisible, put
\begin{equation}\label{eq:rational-tautological}
 \mathcal L_0=\frac1m
 \bigl[\tau_0^*\OO_{Y_k^{\mathrm{GG}}}(-m)\bigr]
 \in\operatorname{Pic}(\widetilde C_0)\otimes_{\mathbb Z}\mathbb Q.
\end{equation}
If $m$ is sufficiently divisible, then for every integer $\ell>0$ one has
$\OO(-\ell m)\simeq\OO(-m)^{\otimes\ell}$, so comparison at a common
multiple makes this class independent of $m$.
The sheaf $\OO_{Y_k^{\mathrm{GG}}}(-1)$ itself need not be invertible;
for $k=1$ it is the usual tautological line on $\Pline(E)$. Note that
these objects belong to the jet space, whereas $\OO_X(-1)$ belongs to the
fixed embedding of $X$.

Since $\tau_0$ is the normalization of $\Gamma$ followed by inclusion,
\[
 \deg\tau_0^*\OO_{Y_k^{\mathrm{GG}}}(m)
 =c_1\bigl(\OO_{Y_k^{\mathrm{GG}}}(m)\bigr)\cdot\Gamma
 =mH_k\cdot\Gamma.
\]
Thus Lemma~\ref{lem:negative-horizontal} gives
\begin{equation}\label{eq:rational-slope}
 \frac{\deg\mathcal L_0}{\deg\rho_0}
 =-\frac{H_k\cdot\Gamma}{\deg\rho_0}
 >\frac{d h_k}{2(n+1)k}>0.
\end{equation}

By Lemma~\ref{lem:weighted-rescaling} there is an affine representative
whose weighted orders are integral after a finite base change, hence a
line bundle carrying the jet; Proposition~\ref{prop:positive} identifies
its rational class with the pullback of $\mathcal L_0$.

We now describe the source of the finite jet. Let $L$ be a line bundle
on a smooth connected projective curve $\widetilde C$, and let
$\widetilde C_{(k)}(L)$ be the $k$-th infinitesimal neighborhood of the
zero section in $\Tot(L)$:
\[
 \widetilde C_{(k)}(L)
 =\Spec_{\widetilde C}\left(\bigoplus_{q=0}^k L^{-q}\right),
\]
where products of total degree greater than $k$ vanish. In a local frame
$\varepsilon$ of $L$, write a vector $v\in L_y$ as $v=t\varepsilon(y)$.
This $t$ is the fiber coordinate on $\Tot(L)$, and the subscheme
$\widetilde C_{(k)}(L)$ is locally cut out by $t^{k+1}=0$.
Write $\pi_L:\Tot(L)\to\widetilde C$ for the projection, and write the
same symbol for its restriction to $\widetilde C_{(k)}(L)$.

\begin{lemma}[An affine lift after finite base change]\label{lem:weighted-rescaling}
Let $\Gamma$, $\tau_0:\widetilde C_0\to Y_k^{\mathrm{GG}}$, and
$\rho_0$ be as in Lemma~\ref{lem:negative-horizontal}. There is a
finite surjective morphism
$\rho_1:\widetilde C\to\widetilde C_0$ from a smooth connected
projective curve. Set
\[
 \rho=\rho_0\circ\rho_1:\widetilde C\to C,
 \qquad
 \tau=\tau_0\circ\rho_1:\widetilde C\to Y_k^{\mathrm{GG}}.
\]
Then there exist an ordinary line bundle $L$ on $\widetilde C$ and a
morphism over $C$
\[
 \jmath:\widetilde C_{(k)}(L)\longrightarrow\mathcal Z^\times.
\]
The morphism $\jmath$ restricts to $s\circ\rho$ on the zero section, and
its positive-order coefficient tuple is nowhere zero. Equivalently,
for every closed point $y\in\widetilde C$, the restriction of $\jmath$
to the fiber of $\widetilde C_{(k)}(L)\to\widetilde C$ over $y$ is
nonconstant in the jet parameter. Weighted projectivization of these
coefficients gives back the morphism
$\tau:\widetilde C\to Y_k^{\mathrm{GG}}$.
\end{lemma}

\begin{proof}
\proofheading{A generic affine representative.}
We first lift $\tau_0$ over the generic point to an affine jet, and make
the weighted orders of its coefficients integral after a finite base
change.

Fix finitely many jet charts $U_\alpha$ covering $C$, and choose one,
say $U_{\alpha_0}$, containing the image of the generic point of
$\widetilde C_0$. We begin with an affine representative in this single
generic chart. Consider the finite surjective morphism from
Lemma~\ref{lem:veronese-polarization},
\[
 \PP^{(n+1)k-1}\longrightarrow
 \PP(1^{n+1},\ldots,k^{n+1}),
 \qquad [u_{i,q}]\longmapsto[u_{i,q}^{\,q}]
\]
as in~\cite[Section~1.2.2]{Dol82}. Choose a closed point in its fiber
over the generic weighted-projective point defined by $\tau_0$.
Its residue field $K$ is finite over $\Bbbk(\widetilde C_0)$. Choose a
homogeneous coordinate tuple $(u_{i,q})$ for this point. Then $(u_{i,q}^{\,q})_{i,q}$
is an affine representative in the $U_{\alpha_0}$ jet coordinates.

For every other chart, define the tuple $b_\alpha=(b_{\alpha,i,q})$ by
transporting this representative with the jet-coordinate transition
maps \eqref{eq:jet-transition}. The cocycle identities make these tuples
compatible, and their coefficients lie in the same finite extension
$K$; the regular transition coefficients are viewed in $K$ through
$\Bbbk(C)\hookrightarrow\Bbbk(\widetilde C_0)\hookrightarrow K$.
After transporting the representative to all the finitely many charts,
adjoin a $q$-th root of every nonzero coefficient $b_{\alpha,i,q}$.
Only finitely many roots are needed, so this is a further finite
extension. Continue to denote the resulting field by $K$, and
let $\rho_1:\widetilde C\to\widetilde C_0$ be the normalization in
$K$. It is finite and surjective. In characteristic zero a normal curve
is smooth, so $\widetilde C$ is smooth; it is also connected and
projective. With $\rho$ and $\tau$ as in the statement, the curve maps
fit into
\[
\begin{tikzcd}[column sep=large,row sep=large]
 \widetilde C
   \arrow[r,"\rho_1"]
   \arrow[rr,bend left=18,"\tau"]
   \arrow[dr,"\rho"'] &
 \widetilde C_0
   \arrow[r,"\tau_0"]
   \arrow[d,"\rho_0"] &
 Y_k^{\mathrm{GG}}
   \arrow[dl,"\pi_k"] \\
 & C .
\end{tikzcd}
\]
Each nonzero $b_{\alpha,i,q}$ is a $q$-th power in
$\Bbbk(\widetilde C)$, so
\[
 \frac{\ord_y(b_{\alpha,i,q})}{q}\in\mathbb Z
\]
for every nonzero coefficient and every $y\in\widetilde C$.

\proofheading{The parameter line and the based jet.}
For $y\in\widetilde C$, choose a jet chart whose base open contains
$\rho(y)$, and put
\begin{equation}\label{eq:weighted-order}
 w_y(\alpha)=\min_{i,q}\frac{\ord_y(b_{\alpha,i,q})}{q}\in\mathbb Z,
 \qquad \ord_y(0)=+\infty.
\end{equation}
For a second chart at $\rho(y)$, each weight-$q$ transition monomial
has order at least $qw_y(\alpha)$, since its coefficient is regular
there. Hence $w_y(\alpha')\geq w_y(\alpha)$, and the inverse transition
gives the reverse inequality. Write $w_y$ for this common value.
The tuple is nonzero in every chart. There are finitely many charts, and
each nonzero rational coefficient has finitely many zeros and poles, so
only finitely many $w_y$ are nonzero. Put
\[
 D=\sum_y w_y[y],
 \qquad L=\OO_{\widetilde C}(D).
\]
Regard $L$ as a subsheaf of the constant sheaf
$\Bbbk(\widetilde C)$. The constant rational function $1$ then
defines a rational section $s_L$ of $L$ whose divisor is $D$:
locally at $y$, if $z$ is a uniformizer, the frame
$\varepsilon_y=z^{-w_y}$ gives $1=z^{w_y}\varepsilon_y$.
Let $V\subset\widetilde C$ be an open set with
$\rho(V)\subset U_\alpha$ on which $L$ has a frame $\varepsilon$,
and let $t$ be the associated fiber coordinate. Write
\[
 s_L=\gamma\varepsilon
\]
on $V$. For every $y\in V$ we then have
$\ord_y(\gamma)=w_y$. Relative to this frame, define
\[
 a_{\alpha,i,q}=\gamma^{-q}b_{\alpha,i,q},
 \qquad
 a_\alpha=(a_{\alpha,i,q})_{i,q}.
\]
By the definition of $w_y$,
\[
 \ord_y(a_{\alpha,i,q})
 =\ord_y(b_{\alpha,i,q})-qw_y\geq0,
\]
and equality holds for at least one coefficient attaining the minimum
in \eqref{eq:weighted-order}. Hence all coefficients $a_{\alpha,i,q}$
are regular on $V$, and at every $y\in V$ at least one satisfies
$a_{\alpha,i,q}(y)\ne0$. Thus the rescaling gives a nonconstant based
jet at every point.

Together with the prescribed center $s\circ\rho$, the relative
based-jet identification \eqref{eq:based-jet-chart} therefore yields a
unique local morphism over $C$
\[
 \jmath_{\alpha,\varepsilon}:
 \widetilde C_{(k)}(L)|_V\longrightarrow\mathcal Z^\times
\]
such that, in the centered étale coordinates,
\[
 z_{\alpha,i}\circ\jmath_{\alpha,\varepsilon}
   =\sum_{q=1}^k a_{\alpha,i,q}t^q
 \quad\text{in }\OO_V[t]/(t^{k+1}),
 \qquad 1\leq i\leq n+1.
\]
Its restriction to the zero section is $s\circ\rho$.
Equivalently, relative to the frame $\varepsilon$, the functions
$a_{\alpha,i,q}$ are precisely the values of the jet coordinates
$x_{\alpha,i,q}$ introduced above.

If $\varepsilon'=u\varepsilon$, then
$a'_{\alpha,i,q}=u^qa_{\alpha,i,q}$ and $t'=u^{-1}t$, so the displayed
coordinate expressions are unchanged. On a jet-chart overlap, weighted
homogeneity gives
$a_{\alpha'}=G_{\alpha\alpha'}(a_\alpha)$; by
\eqref{eq:jet-transition} the corresponding local maps therefore agree
on the overlap, and they glue to a morphism
\[
 \jmath:\widetilde C_{(k)}(L)\longrightarrow\mathcal Z^\times
\]
over $C$. On the dense open set where $s_L$ is regular and nowhere
vanishing we may use $s_L$ as a frame, so that $\gamma=1$ and
$a_\alpha=b_\alpha$; weighted projectivization of the coefficient tuples
therefore agrees with $\tau$ there. Two morphisms from the integral curve
$\widetilde C$ that agree on a dense open set agree everywhere, by
separatedness of $Y_k^{\mathrm{GG}}$.
\end{proof}

\begin{proposition}[Realization of the inverse tautological class]
\label{prop:positive}
For the data of Lemmas~\ref{lem:negative-horizontal} and
\ref{lem:weighted-rescaling}, let $m>0$ be divisible by every $q\leq k$
and such that $\mathcal S_k^{(m)}$ is generated in degree one. Then
\begin{equation}\label{eq:tautological-identification}
 \tau^*\OO_{Y_k^{\mathrm{GG}}}(m)\simeq L^{-m},
\end{equation}
so $[L]=\rho_1^*\mathcal L_0$ in
$\operatorname{Pic}(\widetilde C)\otimes_{\mathbb Z}\mathbb Q$. In particular,
\begin{equation}\label{eq:positive-slope}
 \frac{\deg L}{\deg\rho}
 =-\frac{H_k\cdot\Gamma}{\deg\rho_0}
 >\frac{d h_k}{2(n+1)k}>0.
\end{equation}
\end{proposition}

\begin{proof}
\proofheading{Tautological evaluation.}
Evaluating the weight-$m$ functions on the affine jet $\jmath$ gives a
morphism
\[
 \rho^*(\mathcal S_k)_m\longrightarrow L^{-m}.
\]
On an open set $V$ with $\rho(V)\subset U_\alpha$, with a frame
$\varepsilon$ of $L|_V$ and with $a_\alpha$ the corresponding
coefficient tuple as in Lemma~\ref{lem:weighted-rescaling}, it is given
by
\[
 \rho^*R\longmapsto R(a_\alpha)\varepsilon^{-m},
 \qquad R\in(\mathcal S_k)_m(U_\alpha).
\]
If $\varepsilon'=u\varepsilon$, then
$a'_{\alpha,i,q}=u^qa_{\alpha,i,q}$, and weighted homogeneity gives
\[
 R(a'_\alpha)(\varepsilon')^{-m}
 =R(a_\alpha)\varepsilon^{-m}.
\]
By \eqref{eq:jet-transition} the evaluation is also compatible with a
change of jet chart, so the evaluation map is globally defined.

It is surjective. At every point some $a_{\alpha,i,q}$ is a unit, and,
since $q\mid m$, the section
$x_{\alpha,i,q}^{m/q}\in(\mathcal S_k)_m$ evaluates to a local generator of
$L^{-m}$. Since $\mathcal S_k^{(m)}$ is generated in degree one, the
Veronese construction of Lemma~\ref{lem:veronese-polarization} realizes
$Y_k^{\mathrm{GG}}$ inside $\Pline((\mathcal S_k)_m^\vee)$, with tautological
quotient
\[
 \pi_k^*(\mathcal S_k)_m\twoheadrightarrow
 \OO_{Y_k^{\mathrm{GG}}}(m).
\]
By the Proj universal property~\cite[Tag~01NJ]{Stacks}, the evaluation
quotient defines a morphism
$\widetilde C\to\Pline((\mathcal S_k)_m^\vee)$. Its local coordinates
are $R(a_\alpha)$, so it is the composite of $\tau$ with the Veronese
embedding. Pulling back the universal quotient therefore identifies
\[
 \tau^*\OO_{Y_k^{\mathrm{GG}}}(m)\simeq L^{-m}.
\]
For $k=1$, the coefficient subline $L\hookrightarrow\rho^*E$ identifies
$L$ with $\tau^*\OO_{\Pline(E)}(-1)$; taking the $m$-th tensor power
of the dual identification gives \eqref{eq:tautological-identification}.
Passing to the rational Picard group gives
\[
 [L]=\frac1m[\tau^*\OO_{Y_k^{\mathrm{GG}}}(-m)]
 =\rho_1^*\mathcal L_0
 \quad\text{in }
 \operatorname{Pic}(\widetilde C)\otimes_{\mathbb Z}\mathbb Q.
\]

\proofheading{The degree.}
Finite pullback multiplies degrees by $\deg\rho_1$, so the class
identification gives
\[
 \deg L=(\deg\rho_1)\deg\mathcal L_0,
 \qquad
 \deg\rho=(\deg\rho_1)(\deg\rho_0).
\]
The ratio is unchanged, so \eqref{eq:rational-slope} proves
\eqref{eq:positive-slope}.
\end{proof}

\section{Polynomial realization and the ruled surface}
\label{sec:realization}

The positive inverse tautological class is now represented, after
finite pullback, by the ordinary line bundle $L$ carrying the jet.
The positivity of $L$ forces the high-order ambient coefficients to
vanish, and the surviving coefficients give the polynomial realization.

Fix $\rho$, $L$, and $\jmath$ as in
Lemma~\ref{lem:weighted-rescaling}. By
Proposition~\ref{prop:positive}, $\deg L>0$. For this choice of data at
jet order $k$, put
\[
 r_k=\left\lfloor\frac{a\deg\rho}{\deg L}\right\rfloor.
\]
The subscript records the jet order; the value also depends on the
chosen $\rho$ and $L$. It bounds the orders of the surviving ambient
coefficients and need not be the actual polynomial degree.
A tautological section $\xi$ of $\pi_L^*L$ on $\Tot(L)$ assigns to a
point $(y,v)$, with $v\in L_y$, the vector $v$ itself.
In a local frame $\varepsilon$ of $L$, write $v=t\varepsilon(y)$; then
$\xi=t\varepsilon$.
The same symbol denotes its restriction to $\widetilde C_{(k)}(L)$.
The pullbacks by $\pi_L$ are written explicitly in
\eqref{eq:finite-coefficients} and omitted in the later polynomial
expressions. Through the cone embedding, write
$\rho^*s=(\rho^*f_0,\ldots,\rho^*f_N)$.
On an open $V\subset\widetilde C$ with $\rho(V)\subset U_\alpha$ and a
frame $\varepsilon$ of $L$, the composite of $\jmath$ with the cone
embedding has the ambient expansion
\[
 \rho^*s+\sum_{q=1}^k B_q^{(\varepsilon)}t^q
 \quad\text{modulo }t^{k+1},
 \qquad
 B_q^{(\varepsilon)}\in
 \Gamma\bigl(V,\rho^*A^{\oplus(N+1)}\bigr).
\]
These are the ambient coefficients of \eqref{eq:ambient-jet}, and their
relation to the centered étale coefficients $a_{\alpha,i,q}$ is
generally nonlinear. If
$\varepsilon'=u\varepsilon$, then
$B_q^{(\varepsilon')}=u^qB_q^{(\varepsilon)}$; consequently the local
expressions $B_q^{(\varepsilon)}\otimes\varepsilon^{-q}$ glue to
$\mathbf B_q=(B_{0,q},\ldots,B_{N,q})$, whose components are global
sections of $\rho^*A\otimes L^{-q}$.

A nonconstant jet in the cone might move only in the scalar direction
and thus project to a constant jet in $X$. The next lemma combines the
nonconstancy ensured by Lemma~\ref{lem:weighted-rescaling} with
$\deg L>0$ from Proposition~\ref{prop:positive} to exclude this
possibility; it also establishes the ambient coefficient vanishing.

\begin{lemma}[Coefficient vanishing and nonconstant projection]\label{lem:realization-coefficients}
Let $\rho$, $L$, and $\jmath$ be as in
Lemma~\ref{lem:weighted-rescaling}, so that the positive-order coefficient
tuple of $\jmath$ is nowhere zero, and $\deg L>0$ by
Proposition~\ref{prop:positive}. With $r_k$ as above, composing
$\jmath$ with the cone embedding in
$\Tot(A^{\oplus(N+1)})$ gives $N+1$ coordinate sections
$P_\ell^{(k)}$ of $\pi_L^*\rho^*A$ on $\widetilde C_{(k)}(L)$.
Their expansions are
\begin{equation}\label{eq:finite-coefficients}
 \begin{aligned}
 P_\ell^{(k)}&=\pi_L^*(\rho^*f_\ell)
  +\sum_{q=1}^k\pi_L^*B_{\ell,q}\,\xi^q,\\
  B_{\ell,q}&\in
 H^0\bigl(\widetilde C,\rho^*A\otimes L^{-q}\bigr),
 \end{aligned}
\end{equation}
for $0\leq\ell\leq N$, where the factors $L^{-q}$ and $L^q$ pair
in each product $B_{\ell,q}\xi^q$. Then:
\begin{enumerate}
\item[(i)] $B_{\ell,q}=0$ for every $0\leq\ell\leq N$ and
 $r_k<q\leq k$.
\item[(ii)] The composite
 $\widetilde C_{(k)}(L)\xrightarrow{\jmath}\mathcal Z^\times\to X$
 is nonconstant in the jet parameter over the generic point of
 $\widetilde C$.
\end{enumerate}
In particular, $r_k\geq1$.
\end{lemma}

\begin{proof}
\proofheading{Vanishing of higher coefficients.}
Since $\jmath$ restricts to $s\circ\rho$ on the zero section, the
decomposition
$\pi_{L*}\OO_{\widetilde C_{(k)}(L)}=\bigoplus_{q=0}^kL^{-q}$ gives
\eqref{eq:finite-coefficients}, with each term taking values in
$\pi_L^*\rho^*A$.

The coefficient bundle $\rho^*A\otimes L^{-q}$ has degree
$a\deg\rho-q\deg L$. For $q>r_k$ this degree is negative, so
$B_{\ell,q}=0$.

\proofheading{Nonconstant projection.}
Let $\eta=\Spec K$ be the generic point of $\widetilde C$, where
$K=\Bbbk(\widetilde C)$. Suppose that the projected jet is the composite
\[
 \Spec K[t]/(t^{k+1})\longrightarrow\eta
 \xrightarrow{(f\circ\rho)|_\eta}X,
\]
where the first arrow is the canonical projection.
Then its homogeneous coordinate
tuple would be proportional to $\rho^*s$ modulo $t^{k+1}$. In a local
frame of $L$, comparing coefficients of the coordinate minors gives
\[
 (\rho^*f_i)B_{j,q}-(\rho^*f_j)B_{i,q}=0
 \qquad(0\leq i,j\leq N,\ 1\leq q\leq k)
\]
at the generic point, hence everywhere on $\widetilde C$. Thus
$\mathbf B_q$ is generically proportional to $\rho^*s$. On the open
set where $\rho^*f_i$ is nowhere vanishing,
the ratio $B_{i,q}/\rho^*f_i$ is a regular section of $L^{-q}$.
Generic proportionality gives
\[
 B_{j,q}=(\rho^*f_j)\frac{B_{i,q}}{\rho^*f_i}
 \qquad(0\leq j\leq N)
\]
throughout that open set, since both sides are regular.
The coordinate sections $\rho^*f_i$ have no common zero, so these opens
cover $\widetilde C$; the ratios agree on overlaps and glue to
\[
 \eta_q\in H^0(\widetilde C,L^{-q}),
 \qquad \mathbf B_q=(\rho^*s)\otimes\eta_q.
\]
Thus each purely scalar coefficient comes from $H^0(\widetilde C,L^{-q})$,
which is zero because $\deg L>0$. The jet $\jmath$ would then be the
constant jet based at $s\circ\rho$, contradicting the nowhere-zero
positive-order coefficient tuple of
Lemma~\ref{lem:weighted-rescaling}. The projected jet is therefore
nonconstant at the generic point. Some positive-order coefficient
survives, and the vanishing just proved gives $r_k\geq1$.
\end{proof}

\begin{theorem}[Polynomial realization]\label{thm:realization}
Assume that the fixed jet order $k$ satisfies
\begin{equation}\label{eq:choose-order}
 d h_k>2(n+1)\delta a.
\end{equation}
For the data above, one has
\[
 1\leq r_k<\frac{2(n+1)ka}{d h_k},
 \qquad \delta r_k<k,
\]
and there is a morphism
$\widehat\jmath:\Tot(L)\to\mathcal Z$ over the map
$\rho\circ\pi_L:\Tot(L)\to C$, polynomial of
degree at most $r_k$ in the fiber coordinate, and extending $\jmath$
(viewed as a morphism to $\mathcal Z$). If
\[
 U=\widehat\jmath^{-1}(\mathcal Z^\times),
\]
then $U$ contains the zero section and projectivization gives a morphism
$G:U\to X$. Its restriction to the zero section is $f\circ\rho$, and
the induced rational map $\Tot(L)\dashrightarrow X$ is nonconstant on
a general fiber of $\pi_L$.
\end{theorem}

\begin{proof}
We extend the given finite jet
$\jmath:\widetilde C_{(k)}(L)\to\mathcal Z^\times$ to a polynomial map
from $\Tot(L)$ to $\mathcal Z$, using its coefficient sections.
Lemma~\ref{lem:realization-coefficients} gives $r_k\geq1$ and
the vanishing of all coefficients of order greater than $r_k$.
The slope bound
\eqref{eq:positive-slope} and the choice \eqref{eq:choose-order} give
\begin{equation}\label{eq:equation-degree}
 r_k\leq\frac{a\deg\rho}{\deg L}
 <\frac{2(n+1)ka}{d h_k}<\frac{k}{\delta}.
\end{equation}
The surviving coefficients $B_{\ell,q}$ define polynomial coordinate
sections $P_\ell$ of $\pi_L^*\rho^*A$ on $\Tot(L)$:
\begin{equation}\label{eq:exact-polynomials}
 P_\ell=\rho^*f_\ell+\sum_{q=1}^{r_k}B_{\ell,q}\xi^q.
\end{equation}
Their restrictions to $\widetilde C_{(k)}(L)$ recover the coordinates
of $\jmath$. For each defining equation $F_j$, the section
$F_j(P_0,\ldots,P_N)$ takes values in
$\pi_L^*\rho^*A^{\otimes d_j}$ and has fiber degree at most
$d_jr_k\leq\delta r_k<k$.
In local frames of $L$ and $\rho^*A$ its reduction modulo $t^{k+1}$
is zero, because $\jmath$ takes values in $\mathcal Z$. But a polynomial
of degree at most $k$ whose reduction modulo $t^{k+1}$ is zero is
itself zero. Hence all coefficients vanish in every local
trivialization, and the section vanishes identically. All defining
equations of the cone hold, so the tuple
$(P_0,\ldots,P_N)$ gives the required extension
$\widehat\jmath:\Tot(L)\to\mathcal Z$.

On the zero section the tuple $(P_\ell)$ equals the nowhere-zero tuple
of coordinate sections $(\rho^*f_\ell)$. Hence
\[
 U=\widehat\jmath^{-1}(\mathcal Z^\times)
   =\Tot(L)\setminus V(P_0,\ldots,P_N)
\]
contains the entire zero section, and projectivization gives a morphism
$G=[P_0:\cdots:P_N]:U\to X$ restricting there to $f\circ\rho$.

Since $\widehat\jmath$ extends $\jmath$, a constant projection on the
generic fiber would make the projected generic jet constant, contrary to
Lemma~\ref{lem:realization-coefficients}. The projection is therefore
nonconstant on the generic fiber. Equivalently, some positive-order
coefficient of $P_i(t)P_j(0)-P_j(t)P_i(0)$, written in local frames, is
nonzero at the generic point. Such a coefficient remains nonzero on a
nonempty open subset of $\widetilde C$, so the projection is nonconstant
on a general fiber.
\end{proof}

For fixed $n,a,d$ and arbitrary choices of the auxiliary data at each
sufficiently large jet order, \eqref{eq:equation-degree} gives
$0<r_k/k<2(n+1)a/(d h_k)\to0$ as $k\to\infty$.
It is this uniform control that makes the truncated identities exact.

Nonconstancy has been proved on a general fiber only. On special fibers
the projection may be constant, and $\widehat\jmath$ may meet the cone
vertex outside $U$. After compactification and resolution,
Section~\ref{sec:completion} treats these fibers.

To pass from the polynomial map to complete rational curves, we
compactify its source fiberwise.
The subbundle $L\hookrightarrow\OO_{\widetilde C}\oplus L$ gives a section
of $\Pline(\OO_{\widetilde C}\oplus L)$, whose complement is identified
with $\Tot(L)$ by sending $v\in L_y$ to the line spanned by $(1,v)$.
Under this identification, the section defined by the other
summand $\OO_{\widetilde C}$ is the zero section of $\Tot(L)$.

\begin{corollary}[Ruled-surface realization]\label{cor:ruled-realization}
For the data of Theorem~\ref{thm:realization}, set
$W=\Pline(\OO_{\widetilde C}\oplus L)$, the projective compactification
of $\Tot(L)$. The induced rational map $W\dashrightarrow X$ can be
resolved by point blowups $\beta:S\to W$ whose centers lie away from the
zero section. The resulting smooth connected projective surface has a
morphism
$\pi_S:S\to\widetilde C$ with connected fibers, and its general fiber is
$F\simeq\PP^1$. The zero section lifts to a section
$\sigma:\widetilde C\to S$ of $\pi_S$, so
$\pi_S\circ\sigma=\mathrm{id}_{\widetilde C}$, and there is a morphism
$\Phi:S\to X$ such that $\Phi\circ\sigma=f\circ\rho$.
The reduced support of every
fiber of $\pi_S$ is a connected tree of smooth rational curves, and
\begin{equation}\label{eq:embedding-bound}
 1\leq\deg\bigl(\Phi^*\OO_X(1)|_F\bigr)\leq r_k
 <\frac{2(n+1)ka}{d h_k}.
\end{equation}
This degree is that of the parametrized map $\Phi|_F$, including the
degree onto its image.
\end{corollary}

\begin{proof}
Start with the morphism $G:U\to X$ from
Theorem~\ref{thm:realization}, where $U\subset\Tot(L)\subset W$
contains the entire zero section. It determines a rational map
$\Phi_0:W\dashrightarrow X$ that is regular on $U$. On the generic
ruling fiber, homogenizing \eqref{eq:exact-polynomials} and removing
any common factor gives a morphism of $\OO_X(1)$-degree between one
and $r_k$. On each general ruling fiber the same homogeneous
coordinates have degree at most $r_k$, and removing their common
factor can only decrease this degree.

Let $\pi_W:W\to\widetilde C$ be the ruling. Since $W$ is a regular
projective surface and $X$ is proper, elimination of indeterminacy
\cite[Tag~0C5H]{Stacks} resolves $\Phi_0$ by a sequence of blowups at
closed points that do not lie over the given open set $U$. In
particular, the composition of these blowups is an isomorphism over $U$,
so the zero section is untouched. Let $\beta:S\to W$ be the resulting
composition, write $\Phi:S\to X$ for
the resolved morphism, and put $\pi_S=\pi_W\circ\beta$. These maps are
summarized by
\[
\begin{tikzcd}[column sep=large,row sep=large]
 S
   \arrow[r,"\beta"]
   \arrow[rr,bend left=18,"\Phi"]
   \arrow[dr,"\pi_S"'] &
 W
   \arrow[r,dashed,"\Phi_0"]
   \arrow[d,"\pi_W"] &
 X \\
 & \widetilde C .
\end{tikzcd}
\]
The surface $S$ is smooth, connected, and projective. Since $\beta$
is an isomorphism over $U$, the zero
section lifts to a section $\sigma$ satisfying
\[
 \Phi\circ\sigma=f\circ\rho.
\]
The reduced support of each fiber is a simple normal crossings divisor
of smooth rational components with a tree as its dual graph. This holds
initially on $W$. Inductively, a blowup at a smooth point of a fiber
component adds a new exceptional $\PP^1$ as a leaf of the dual graph,
while a blowup at a node inserts the exceptional $\PP^1$ and subdivides
the corresponding edge. Thus the strict transforms and the exceptional
curves remain smooth rational curves, the crossings remain simple, and
the dual graph remains a tree. The scheme-theoretic fiber may have
component multiplicities.
A general fiber avoids the finitely many blowup centers, so it remains
$F\simeq\PP^1$, with the degree bound just proved. The estimate for
$r_k$ in Theorem~\ref{thm:realization} gives \eqref{eq:embedding-bound}.
\end{proof}

\section{Rational curves through every prescribed closed point}
\label{sec:completion}

We can already produce rational curves through prescribed points. Since
$\rho$ is surjective, every prescribed closed point of $f(C)$ occurs as
$(f\circ\rho)(y)$ for some $y\in\widetilde C$, so it is enough to treat a
closed point $y$ of $\widetilde C$. If $\Phi$ is nonconstant on the
component meeting the zero section, that component gives the desired
curve. Otherwise, positivity of the total fiber degree guarantees a
nonconstant component elsewhere; connectedness of the fiber support
then carries the prescribed value to the first such component. The lemma
below makes this argument precise.

Its proof uses the same specialization principle as the proof of
\cite[Lemma~7.8]{Deb}, where a bounded-degree rational curve
degenerates to rational components of no larger total degree.

\begin{lemma}[Prescribed-point specialization]\label{lem:prescribed-point}
In the setting of Corollary~\ref{cor:ruled-realization}, put
\[
 A_S=\Phi^*\OO_X(1),\qquad d_F=A_S\cdot F,
\]
where $F$ is a smooth general fiber. For every closed point $y\in\widetilde C$
there is a nonconstant morphism $b_y:\PP^1\to X$ such that
\[
 b_y(0)=(f\circ\rho)(y),\qquad
 \deg b_y^*\OO_X(1)\leq d_F\leq r_k.
\]
One may take for $b_y$ the restriction of $\Phi$ to a noncontracted
component of $\pi_S^*(y)$, composed with an identification of that
component with $\PP^1$.
\end{lemma}

\begin{proof}
The line bundle $A_S$ is nef, and nonconstancy of $\Phi|_F$ gives
$1\leq d_F\leq r_k$. Fix $y\in\widetilde C$ and write the
scheme-theoretic fiber as
\[
 \pi_S^*(y)=\sum_i m_i\Gamma_i,
 \qquad m_i\in\mathbb Z_{>0}.
\]
By the construction in Corollary~\ref{cor:ruled-realization}, its reduced
support is connected and every $\Gamma_i$ is a smooth rational curve.
All fibers of $\pi_S$ are numerically equivalent: points of
$\widetilde C$ are numerically equivalent, and their pullbacks are the
fiber divisors. Consequently
\begin{equation}\label{eq:special-fiber-degree}
 \sum_i m_i(A_S\cdot\Gamma_i)=A_S\cdot F=d_F>0.
\end{equation}
Every summand is nonnegative, so some $\Gamma_i$ has positive
$A_S$-degree; equivalently, $\Phi|_{\Gamma_i}$ is nonconstant.

Choose a chain from a fiber component containing $\sigma(y)$
to a nonconstant component, and stop at the first nonconstant one.
If this is the initial component, its image contains $(f\circ\rho)(y)$.
Otherwise the preceding components are constant, and their values
agree at successive intersections. Each value is therefore
$(\Phi\circ\sigma)(y)=(f\circ\rho)(y)$, so the last intersection gives
a point of the nonconstant component with that image. Choose an
isomorphism from $\PP^1$ to that component sending $0$ to such a point.
Composing with $\Phi$ gives the required
morphism $b_y:\PP^1\to X$.

Equation~\eqref{eq:special-fiber-degree} also gives
$\deg b_y^*\OO_X(1)\leq d_F\leq r_k$, as asserted.
\end{proof}

\begin{proof}[Proof of Theorem~\ref{thm:main}]
Since $h_k\to\infty$, let $k$ be the least positive integer with
$dh_k>2(n+1)\delta a$. The elementary estimate
$h_k\geq\log(k+1)$ shows, for example, that
\begin{equation}\label{eq:jet-order-upper-bound}
 k\leq
 \left\lceil\exp\!\left(\frac{2(n+1)\delta a}{d}\right)\right\rceil.
\end{equation}
Apply Corollary~\ref{cor:ruled-realization}, and fix this construction
before choosing $x$. The finite morphism
$\rho:\widetilde C\to C$ is surjective, so for any $x\in f(C)$ one can
choose $y\in\widetilde C$ with $(f\circ\rho)(y)=x$. By
Lemma~\ref{lem:prescribed-point}, there is a nonconstant morphism
$\PP^1\to X$ through $x$.
\end{proof}

\section{Numerical decomposition and its consequences}
\label{sec:numerical}

The zero section records the original curve, and the ruling fibers
produce rational curves. Comparing the zero and infinity sections gives
an effective numerical decomposition relating these two classes.

Retain the data of Corollary~\ref{cor:ruled-realization}.
In this section, an integral curve in a cycle expression is identified
with its associated one-cycle. Put
$Z_C=f_*[C]$, $\mu=\dfrac{\deg L}{\deg\rho}\in\mathbb Q_{>0}$, and
$Z_F=\Phi_*F$, where $F$ is a smooth general fiber of
$\pi_S:S\to\widetilde C$. The estimate
\eqref{eq:positive-slope} gives
\[
 \mu>\frac{dh_k}{2(n+1)k}.
\]

Numerical equivalence of one-cycles is denoted by $\equiv$: one has
$Z\equiv Z'$ if and only if $D\cdot Z=D\cdot Z'$ for every Cartier divisor $D$.
Write $N_1(X)_{\mathbb R}=Z_1(X)_{\mathbb R}/\!\equiv$ and $[Z]$ for the
numerical class of $Z$. The closed cone of curves is
\[
 \NE(X)=\overline{\operatorname{Cone}\{[C]\mid C\subset X
 \text{ an integral curve}\}}\subset N_1(X)_{\mathbb R}.
\]

An effective $\mathbb Q$-cycle means a finite sum
\[
 Z=\sum_i m_i C_i,\qquad m_i\in\mathbb Q_{\geq0},
\]
where the $C_i$ are arbitrary integral curves on $X$. An effective
cycle has integer coefficients; an effective cycle of rational curves
has, in addition, rational irreducible components. We allow $Z=0$
in each case.

For a proper morphism $\Phi:S\to X$ and an integral curve
$\Gamma\subset S$ with curve image $R=\Phi(\Gamma)$, we use the
cycle-theoretic pushforward
\[
\Phi_*\Gamma
=
[\Bbbk(\Gamma):\Bbbk(R)]\,R,
\]
while $\Phi_*\Gamma=0$ if $\Gamma$ is contracted. In particular,
$Z_F=e_FR_F$, where $R_F=\Phi(F)_{\mathrm{red}}$ is an integral rational
curve and $e_F=\deg(F\to R_F)\geq1$.

To compare the two sections, recall the blowup map
$\beta:S\to W$ from Corollary~\ref{cor:ruled-realization},
and write $\pi_W:W\to\widetilde C$ for the ruling.
The zero and infinity sections are
$\Sigma_0=\Pline(\OO_{\widetilde C})$ and $\Sigma_\infty=\Pline(L)$,
viewed as subvarieties of $W=\Pline(\OO_{\widetilde C}\oplus L)$.
Let $\widetilde\Sigma_\infty$ be the strict transform of $\Sigma_\infty$ in $S$.
The total transform of $\Sigma_\infty$ is
\[
 \beta^*\Sigma_\infty
   =\widetilde\Sigma_\infty+\sum_j a_jE_j,
 \qquad a_j\in\mathbb Z_{\geq0},
\]
where the $E_j$ are irreducible exceptional curves and the $a_j$
include the multiplicities from successive blowups. Put
\[
 Z_\infty=\Phi_*\widetilde\Sigma_\infty,
 \qquad Z_{\mathrm{exc}}=\sum_j a_j\Phi_*E_j.
\]

\begin{proposition}[An effective numerical decomposition]\label{prop:decomposition}
The cycle $Z_\infty$ is effective, $Z_{\mathrm{exc}}$ is an effective cycle
of rational curves, and
\begin{equation}\label{eq:effective-decomposition}
 (\deg\rho)Z_C\equiv Z_\infty+(\deg L)Z_F+Z_{\mathrm{exc}}.
\end{equation}
The cycle $Z_F$ is nonzero and has rational curve components. For every
closed point $x\in f(C)$, there are an integral rational curve $R_x$
through $x$ and an effective $\mathbb Q$-cycle $Q_x$ such that
\begin{equation}\label{eq:initial-domination}
 Z_C\equiv\mu R_x+Q_x.
\end{equation}
In particular, $[Z_C]-\mu[R_x]\in\NE(X)$.
\end{proposition}

\begin{proof}
We first compare the two sections $\Sigma_0$ and $\Sigma_\infty$ on
$W$; the relation between them becomes a cycle relation on $X$ once it
is pulled back to $S$ and pushed forward by $\Phi$.
The projections from the tautological subline $\OO_W(-1)$ to
$\OO_W$ and $\pi_W^*L$
vanish along $\Sigma_\infty$ and $\Sigma_0$, respectively. Hence
\[
 \OO_W(\Sigma_\infty)\simeq\OO_W(1),
 \qquad
 \OO_W(\Sigma_0)\simeq\OO_W(1)\otimes\pi_W^*L.
\]
Taking their quotient gives
\[
 \OO_W(\Sigma_0-\Sigma_\infty)\simeq \pi_W^*L,
 \qquad \Sigma_0\equiv\Sigma_\infty+(\deg L)F_W,
\]
where $F_W$ is a ruling fiber. All blowup centers avoid
$\Sigma_0$, so $\beta^*\Sigma_0=\sigma(\widetilde C)$, and
\[
 \Phi_*\sigma(\widetilde C)=(f\circ\rho)_*\widetilde C
   =(\deg\rho)Z_C.
\]
The total pullback of a ruling fiber is numerically equivalent to $F$.
Pulling back the section relation and pushing forward by $\Phi$
therefore gives \eqref{eq:effective-decomposition}.
The projection formula preserves numerical equivalence, and contracted
components have zero pushforward. The noncontracted images of
$E_j\simeq\PP^1$ and of $F\simeq\PP^1$ are rational, and
$\Phi|_F$ is nonconstant. The components of $Z_\infty$ need not be
rational curves.

For $x\in f(C)$, choose $y\in \widetilde C$ with
$(f\circ\rho)(y)=x$. By Lemma~\ref{lem:prescribed-point}, choose a
noncontracted component $\Gamma_i$ of $\pi_S^*(y)$ whose reduced image
$R_x$ contains $x$. If $m_i$ is the multiplicity of $\Gamma_i$ in the
scheme-theoretic fiber, then the coefficient of $R_x$ in
\[
 Z_y=\Phi_*\bigl(\pi_S^*(y)\bigr)\equiv Z_F
\]
is at least $m_i\deg(\Phi|_{\Gamma_i}:\Gamma_i\to R_x)\geq1$.
Hence $Z_y-R_x$ is effective, and the right-hand side of
\[
 Z_C-\mu R_x
 \equiv \frac{1}{\deg\rho}(Z_\infty+Z_{\mathrm{exc}})
       +\mu(Z_y-R_x)
\]
is the required effective $\mathbb Q$-cycle $Q_x$.
\end{proof}

We can now shorten the rational curve, measured by the degree with
respect to a fixed ample divisor $H$, by two-point bend-and-break.
The point of the argument is that replacing a curve by a component of a
broken effective cycle keeps the effective remainder, with the same
coefficient $\mu$. Only the effective domination already obtained, and
the smoothness and projectivity of $X$, are used.

A real Cartier divisor $P$ is nef if $P\cdot R\geq0$ for every integral
curve $R\subset X$. A face $\mathcal F$ of $\NE(X)$ is a subcone such
that $u,v\in\NE(X)$ and $u+v\in\mathcal F$ imply
$u,v\in\mathcal F$.

\begin{corollary}[Domination after bend-and-break]\label{cor:dominated}
For every closed point $x\in f(C)$, there are an integral rational
curve $R_x$ through $x$ and an effective $\mathbb Q$-cycle $Q_x$ with
\begin{equation}\label{eq:dominated-short}
 Z_C\equiv\mu R_x+Q_x,
 \qquad -K_X\cdot R_x\leq n+1.
\end{equation}
Moreover, for every nef real Cartier divisor $P$ one has
\[
 P\cdot R_x\leq\frac{P\cdot Z_C}{\mu}
 \leq\frac{2(n+1)k}{dh_k}\,P\cdot Z_C,
\]
and the class of $R_x$ belongs to every face of $\NE(X)$ containing
$[Z_C]$.
\end{corollary}

\begin{proof}
Fix an ample Cartier divisor $H$. Choose an integral rational curve
$R$ through $x$ of smallest $H$-degree among those for which
$[Z_C]-\mu[R]$ is represented by an effective $\mathbb Q$-cycle.
Such curves exist by Proposition~\ref{prop:decomposition}, and their
$H$-degrees are positive integers. Suppose that
$-K_X\cdot R>n+1$. Parametrize the normalization by
$\nu:\PP^1\to R\subset X$, with $\nu(0)=x$ and
$\nu(\infty)\ne x$, so that $\nu_*[\PP^1]=R$.
Since $-K_X\cdot R$ is an integer, it is at least $n+2$.
Write $\operatorname{Hom}(\PP^1,X)$ for the
scheme of morphisms. Adding point conditions after a semicolon denotes
the closed subscheme of maps with those prescribed values, before
quotienting by reparametrization. Deformation theory and
Riemann--Roch give the local dimension bound
\[
 \begin{aligned}
 &\dim_{[\nu]}\operatorname{Hom}
 \bigl(\PP^1,X;0\mapsto x,\infty\mapsto\nu(\infty)\bigr)\\
 &\qquad\geq\chi(\PP^1,\nu^*T_X)-2n
 =-K_X\cdot R-n\geq2.
 \end{aligned}
\]
The automorphisms of $\PP^1$ fixing $0$ and $\infty$ form a
one-dimensional group, so beyond reparametrization the bound still
provides deformations.
Two-point bend-and-break~\cite[\textup{(6.2)} and Proposition~7.3]{Deb}
therefore gives a connected, nonintegral effective one-cycle $Z'$ with
rational components, numerically equivalent to $R$ and passing through
both fixed image points $x$ and $\nu(\infty)$. Choose an integral
component $R'$ of $Z'$ through $x$. Its $H$-degree is strictly smaller
than $H\cdot R$: this is immediate if $Z'$ has more than one component,
and it also holds if $Z'$ is a nontrivial multiple of a single component.
Let $Z_{\mathrm{res}}$ be an effective $\mathbb Q$-cycle representing
$[Z_C]-\mu[R]$. Then $R'$ contradicts minimality, since
\[
 [Z_C]-\mu[R']=[Z_{\mathrm{res}}]+\mu[Z'-R'],
\]
and $Z_{\mathrm{res}}+\mu(Z'-R')$ is an effective $\mathbb Q$-cycle.
This proves \eqref{eq:dominated-short}.

The simultaneous inequalities follow by intersecting with a nef divisor
and inserting the lower bound for $\mu$. If $[Z_C]$ lies in a face
$\mathcal F$, then its decomposition as
$\mu[R_x]+([Z_C]-\mu[R_x])$ is a sum of two classes in
$\NE(X)$, so the defining property of a face gives
$[R_x]\in\mathcal F$.
\end{proof}

Because $\mu$ was fixed before $x$ was chosen, the same coefficient
works for every closed point. This completes the proof of
Theorem~\ref{thm:effective-domination}.

The anticanonical degree need not be positive in general. The
constant-slope condition below is an additional hypothesis on the
chosen face; it is automatic for a one-dimensional face containing
$[Z_C]$, since $d>0$.

\begin{corollary}[A face of constant slope]\label{cor:face}
Let $H$ be an ample Cartier divisor, and let $\mathcal F$ be a
face of $\NE(X)$ containing $[Z_C]$. Suppose that
$-K_X\cdot\gamma=\lambda H\cdot\gamma$ for every
$\gamma\in\mathcal F$, with $\lambda>0$ fixed.
Then every closed point $x\in f(C)$ lies on an integral rational curve $R_x$
whose class belongs to $\mathcal F$ and satisfies
\[
 0<-K_X\cdot R_x\leq n+1,
 \qquad H\cdot R_x\leq\frac{n+1}{\lambda}
       =(n+1)\frac{H\cdot Z_C}{d}.
\]
\end{corollary}

\begin{proof}
Apply Corollary~\ref{cor:dominated}. The slope identity on
$\mathcal F$ gives $-K_X\cdot R_x=\lambda H\cdot R_x>0$ together with
$d=\lambda H\cdot Z_C$, so both bounds follow.
\end{proof}

This applies, in particular, when $[Z_C]$ spans a $K_X$-negative
extremal ray.

\begin{remark}[Fano varieties]
For a smooth Fano variety and a prescribed closed point $x$, take the
normalization of any integral projective curve through $x$.
Ampleness of $-K_X$ gives the required positive anticanonical degree,
so Corollary~\ref{cor:dominated} yields
$0<-K_X\cdot R_x\leq n+1$ through every such point.
\end{remark}

\section{The projective uniruledness criterion}\label{sec:bdpp}

A divisor class is pseudoeffective if it lies in the closure of
the cone of effective divisor classes in $N^1(X)_{\mathbb R}$, the
real vector space of Cartier divisors modulo numerical equivalence.
The variety $X$ is uniruled if a family of nonconstant maps from $\PP^1$,
parametrized by a finite-type variety, has dominant evaluation.

When $K_X$ is not pseudoeffective, BDPP cone duality yields a covering
family of curves of negative canonical degree. To turn the pointwise
construction into a covering
family of rational curves, we use the uniform embedding-degree bound of
Corollary~\ref{cor:ruled-realization} together with its specialization
in Lemma~\ref{lem:prescribed-point}.

\begin{proof}[Proof of Theorem~\ref{thm:bdpp}]
Suppose that $K_X$ is not pseudoeffective. Fix a projective embedding
$X\hookrightarrow\PP^N$, write $n=\dim X$, and let $\delta$ bound the
degrees of the defining equations, as in Section~\ref{sec:weighted-jets}.
If $n=1$, then $\deg K_X<0$; applying Theorem~\ref{thm:main} to
$\mathrm{id}_X$ gives a nonconstant morphism $\PP^1\to X$, whose image
is $X$. For the rest of this implication, assume $n\geq2$.

By the strongly movable form of BDPP cone duality
\cite[Definition~1.3(v) and Theorem~2.2]{BDPP13}, there are a smooth
projective variety $\widehat X$, a birational morphism
$\varpi:\widehat X\to X$, and very ample divisors
$A_1,\ldots,A_{n-1}$ on $\widehat X$ such that
\[
 \alpha=\varpi_*(A_1\cdots A_{n-1}),\qquad K_X\cdot\alpha<0.
\]
Indeed, nonnegativity on every such generating class would imply
nonnegativity on the strongly movable cone and its closure.
By Bertini's theorem, a nonempty open subset
$T\subset |A_1|\times\cdots\times|A_{n-1}|$ parametrizes smooth
integral complete-intersection curves $C_t$ meeting the locus where
$\varpi$ is an isomorphism. Put $f_t=\varpi|_{C_t}:C_t\to X$.
These maps are nonconstant and have dominant total evaluation.
To see the latter assertion, the incidence variety over the full product
of linear systems is a fiber product of projective bundles over
$\widehat X$, hence is irreducible and surjects onto $\widehat X$.
Its restriction to the nonempty open set $T$ still has dense image,
and the same holds after composition with $\varpi$.

Since $(f_t)_*[C_t]\equiv\alpha$, the integers
\[
 a_0=\OO_X(1)\cdot\alpha>0,\qquad
 d_0=-K_X\cdot\alpha>0
\]
are independent of $t$. The total evaluation is a finite-type algebraic
morphism, so its constructible dense image contains a dense open subset
of $X$. Let $k$ be the least positive integer with
$d_0h_k>2(n+1)\delta a_0$, and apply the construction to the maps $f_t$.
By \eqref{eq:embedding-bound} and Lemma~\ref{lem:prescribed-point},
every closed point of each $f_t(C_t)$ lies on the image of a nonconstant
morphism of $\OO_X(1)$-degree at most
\[
 M=\left\lceil\frac{2(n+1)ka_0}{d_0h_k}\right\rceil-1.
\]
Indeed, the constructed degree is a positive integer strictly smaller
than $2(n+1)ka_0/(d_0h_k)$, which gives the displayed bound and
$M\geq1$. The integer $M$ depends only on $n$, on the fixed embedding
and the equation degree bound $\delta$, and on $a_0,d_0$; it does not
depend on the genus or on $t$.
The degree of the additional finite cover cancels in $\deg L/\deg\rho$,
as in Proposition~\ref{prop:positive}, so no uniform bound on the
cover degrees is needed. For any resulting map $b:\PP^1\to X$ with
reduced image $R$, the projection formula gives
\[
 \deg b^*\OO_X(1)=\deg(\PP^1\to R)\,\bigl(\OO_X(1)\cdot R\bigr).
\]
Thus the integral rational curve $R$ has degree at most $M$, and its
normalization gives a morphism of degree $\OO_X(1)\cdot R$.
In this way we obtain such maps through all closed points of a dense
open subset of $X$. The choices made for different $C_t$ need not vary
algebraically with $t$: the uniform degree bound places all the resulting
maps in a fixed finite union of parameter spaces.

\begingroup\postdisplaypenalty=10000
For $1\leq e\leq M$, let $\operatorname{Hom}_e(\PP^1,X)$ parametrize
the morphisms of $\OO_X(1)$-degree $e$. These schemes are of finite
type~\cite[\S6.1]{Deb}. The evaluation morphism from
\[
 \PP^1\times\coprod_{e=1}^M\operatorname{Hom}_e(\PP^1,X)
 \longrightarrow X
\]
has dense image. Its source has finitely many irreducible components,
so one of them has dominant evaluation. Thus $X$ is uniruled.
\par\endgroup

Conversely, suppose that $X$ is uniruled. In characteristic zero,
uniruledness is separable, so we may choose an irreducible covering
family of free maps $b:\PP^1\to X$ in $\operatorname{Hom}(\PP^1,X)$
\cite[Corollary~9.13(b)]{Deb}. Let $\gamma=[b_*[\PP^1]]$ be their
common numerical class. Freeness means that $b^*T_X$ is globally
generated, so in the splitting
$b^*T_X=\bigoplus_i\OO_{\PP^1}(a_i)$ one has $a_i\geq0$. The nonzero
differential $\OO_{\PP^1}(2)\to b^*T_X$ forces some $a_i\geq2$, so
$K_X\cdot\gamma=-\deg b^*T_X\leq-2$.
For every effective divisor $D$, a general member of the family is not
contained in the support of $D$, so $D\cdot\gamma\geq0$. By continuity,
every pseudoeffective divisor class has nonnegative intersection with
$\gamma$. Hence $K_X$ is not pseudoeffective.
\end{proof}

\section*{Acknowledgments}
We are grateful to Mihai P\u{a}un for his careful reading of an earlier
draft and for his insightful questions and suggestions, which greatly
improved both our understanding of the subject and the exposition of
this paper.

The authors would like to express their sincere gratitude to Bin Dong,
Guoxiong Gao, Jiedong Jiang, Shurui Liu, Zeming Sun, and
Bin Wu for their substantial contributions to the development and refinement of Pharos.

The authors would also like to thank Jihao Liu, Bohan Fang, Jingjun Han, Guchuan Li, Ruochuan Liu, Yujie Luo, Zhenfu Wang, and Yijun Yuan for their valuable feedback and suggestions, which helped improve Pharos.

The authors are grateful to Wanyi He for a careful
check of the Lean formalization, and warmly thank Axel
Delaval and Zhiyuan Zhang for generously
sharing their experience with formalization in Lean.

The authors are grateful to Leheng Chen and Zhenyu Liao for their assistance with the computational infrastructure supporting our experiments.

\appendix

\clearpage

\section{How the problem reached Pharos}\label{app}

\begin{center}
by Bin Guo and Song-Yan Xie
\end{center}

Bin Guo and Song-Yan Xie were looking for an analytic proof of the Miyaoka--Mori criterion. Earlier conversations with AI systems had led them down several unsuccessful paths, recorded in ZIP archives. The archives were meant to serve as a map of roads not to take. Guo and Xie still believed the problem could be solved and had an idea of their own, as yet untested. Before trying it, they decided to put Pharos to the test, though they doubted that an AI system would find a proof. They contacted Guoxiong Gao and asked him to submit the archives with the problem, hoping at least to spare Pharos the routes they had already explored.

About nineteen hours into the run, the Pharos team sent them an AI-generated progress report. Guo and Xie were baffled: it proposed precisely the approaches the archives were supposed to rule out. Had Pharos failed even to take the hint? They suggested stopping the run before it spent more tokens on familiar dead ends. The team, however, followed its own stopping rule: while Pharos still saw a promising direction, they kept it running.

An hour later, the team reported a proof.

Only afterward did Guo and Xie learn that there had been no hint for Pharos to take. The ZIP archives had accidentally been left out of the first run; Pharos had received only the problem. In a separate experiment with the archives included from the outset, it reached essentially the same jet-based proof in about ten hours.

\clearpage

\section{Use of generative AI}\label{app:ai}
\begin{center}
by Bin Dong\footnote{Beijing International Center for Mathematical
Research and Center for Machine Learning Research, Peking University,
Beijing 100871, China. Email: \texttt{dongbin@math.pku.edu.cn}},
Guoxiong Gao\footnote{School of Mathematical Sciences, Peking
University, Beijing 100871, China. Email:
\texttt{samggx@stu.pku.edu.cn}}, Zeming Sun\footnote{Beijing
International Center for Mathematical Research, Peking University,
Beijing 100871, China. Email: \texttt{zeming@kurims.kyoto-u.ac.jp}},
and Bin Wu\footnote{School of Mathematics, Tianjin University,
Tianjin 300354, China. Email: \texttt{s-wb25@bza.edu.cn}}
\end{center}

The main result of this paper, the algebraic construction of rational
curves, was found by Pharos, an automated mathematical reasoning system designed in light of prior experience with Danus~\cite{Liu26}. Guoxiong Gao led the
design of Pharos; for the problem treated here, Zeming
Sun, one of its designers, ran the system and analyzed how it reached
the result. Bin Guo and Song-Yan Xie supplied the problem,
the first question in Appendix~F of Du--Guo--Xie~\cite{DGX26}: for a
nonconstant map $f:C\to X$ from a smooth projective curve, can
$-K_X\cdot f_*[C]>0$ alone produce a rational curve by continuing
disks over $C$ under a strict area bound? Working fully autonomously and without human intervention, Pharos
answered the question as posed.

Bin Guo and Song-Yan Xie undertook to understand, rewrite, expand, and
develop the argument that Pharos produced, and to check its
mathematical correctness. No step of the proof relies on a numerical
computation or on the system's internal acceptance decision, and they
take full responsibility for the mathematical content and
presentation of the main text of the paper.

The proof has also been checked in Lean~4\footnote{The Lean sources
are available at \url{https://github.com/frenzymath/MiyaokaMori-CharZero}.}:
using SHEAF\footnote{SHEAF will
be released at \url{https://github.com/frenzymath/SHEAF}.} (Scalable
Hierarchical Engine for Autonomous Formalization), a system of
autonomous agents built and run by Bin Wu and described at the end of this
appendix, the argument of
Sections~\ref{sec:weighted-jets}--\ref{sec:completion} was formalized
over Mathlib with no human intervention in the proving process.

Pharos is a reasoning system designed for long-horizon research on
challenging, open-ended mathematical problems for which no reasonably
workable route is known in advance. It inherits selected design principles from Danus and
Rethlas~\cite{Ju26}, especially the fact graph and the multi-agent architecture,
while substantially redesigning and extending the remaining system components.

In Pharos, several core components are introduced to assist the main agent, both in mathematics and in global strategy.
We provide helper agents at its disposal: they assess
the workers' progress, collect and digest the literature, and scout
candidate routes before a worker is committed to one, while their
reports inform the main agent without ever entering a proof.
Freed in
this way from much of its context burden, the main agent can
concentrate on the mathematics and on the overall strategy, which it
maintains on a route board, a new central data structure that records
the state of every route. It begins by surveying the techniques that
the literature has applied to the problem, so that every known
approach is tried before new ones are invented. It then monitors each
route and judges whether it is still making real progress or has
stalled and should be stopped, and from what it observes it proposes new routes. The data
structures, skills, control loop, and tools built around this logic
give the system the ability to climb a smooth cliff face: it misses no
foothold hidden in a known approach, finds new ones, and keeps them
moving.

We further equip Pharos with several additional capabilities. Workers reason more deeply, and a
fact has grown from a local lemma or example into a substantial
self-contained advance, which lets a model as capable as GPT-6 Astra be
used to the full. With proper scheduling of computing resources, Pharos
can now write code and run local numerical computations. Its interface with human operators is designed for close collaboration: a new writing system produces
shorter and clearer manuscripts, and produces them faster, and new
skills and a reporter agent let a person follow a run and make the
system follow instructions more faithfully, which favors closer
collaboration. Finally, the command-line tools that schedule its agents are designed to make the system run smoothly and its workings more transparent, and
the context passed to the models during tool use is compressed without
loss of information, which reduces cost and raises throughput. Together
these changes make Pharos markedly more effective and easier to use.

In the run reported here, agents of every role ran on GPT-6 Astra,
with the main agent at effort level ``ultra'', three workers at
``xhigh'' and four at ``max'', and the verifier at ``max''. It took 19 hours and 57 minutes
from the creation of the project to the acceptance of the target, and
a further 1 hour and 52 minutes to write and check its own manuscript.
Its
token usage is recorded in Table~\ref{tab:tokens}; in all, the run
consumed slightly less than twice the weekly usage limit of a \$200
ChatGPT Pro subscription.
The final fact graph contains 273 verified facts. One further fact was
revoked during the run for a quantifier error and independently
reproved. The global memory contains 282 conclusions, 75 identified
obstacles, 305 directions, 420 proof attempts, 263 plans, 77
counterexamples, 82 recorded dead ends, and 382 verification records,
of which 97 are rejections. Facts are long: the median statement has
about 660 words and the median proof about 2200 words.
Table~\ref{tab:closure} records how the 273 facts divide: 4 lie in
the supporting closure of the final target fact, computed from the
dependencies recorded with each fact, and the remaining 269 lie outside
it, unused by the paper.

\begin{table}[ht]
\centering\small
\begin{tabular}{lrrrr}
\hline
& Sessions & Input (M) & of which cached (M) & Output (M)\\
\hline
Main agent & 1 & 435.0 & 427.5 & 1.3\\
Workers & 264 & 1349.0 & 1295.4 & 10.4\\
Verifier & 378 & 433.3 & 390.2 & 6.4\\
Helper agents & 94 & 254.4 & 232.6 & 1.9\\
Manuscript and reports & 14 & 44.5 & 41.1 & 0.3\\
Total & 751 & 2516.3 & 2386.7 & 20.3\\
\hline
\end{tabular}
\vspace{6pt}
\caption{Token usage of the run, in millions of tokens.}
\label{tab:tokens}
\end{table}

\begin{table}[ht]
\centering\small
\begin{tabular}{lrr}
\hline
& Count & Share of all facts\\
\hline
Main-theorem closure & 4 & 1.5\%\\
Outside the closure, unused by the paper & 269 & 98.5\%\\
Total & 273 & 100\%\\
\hline
\end{tabular}
\vspace{6pt}
\caption{Distribution of facts in the fact graph.}
\label{tab:closure}
\end{table}

Table~\ref{tab:inclosure} lists the four facts in the closure. Pharos
first established fact A in a relatively early research task. Then,
after a few hours, Pharos applied fact A on another route and obtained
fact B. After a much longer period of research, Pharos obtained fact
C and, combining it with fact B, obtained the final result, fact D.

\begin{table}[ht]
\centering\small
\begin{tabular}{@{}lp{0.26\textwidth}p{0.34\textwidth}p{0.12\textwidth}p{0.11\textwidth}@{}}
\hline
 & Corresponding statements in the paper & Function in the proof & Dependency & Time since run start\\
\hline
Fact A & Corollary~\ref{cor:ruled-realization}. &
Resolves the polynomial map to a morphism from a ruled surface with
the seed as a section. &
--- &
1:27\\[2pt]
Fact B & Lemma~\ref{lem:prescribed-point} and its proof,
Section~\ref{sec:completion}. &
Specializes a fiber of the ruled surface to obtain a rational curve
through each point of the seed; the final step shared with
bend-and-break. &
fact A &
5:42\\[2pt]
Fact C & Proposition~\ref{prop:harmonic} through
Theorem~\ref{thm:realization}. &
Produces, from $d>0$ alone, jets on a line bundle of positive degree
after a finite cover, and makes them an exact polynomial map; the
replacement for Frobenius amplification. &
--- &
19:46\\[2pt]
Fact D & Theorem~\ref{thm:main} and its proof. &
Composing fact C with fact B to obtain the main result. &
facts B, C &
19:57\\
\hline
\end{tabular}
\vspace{6pt}
\caption{The four facts in the main-theorem closure. The column
``Dependency'' lists the direct dependencies recorded in the fact
graph.}
\label{tab:inclosure}
\end{table}

These four facts are only a small part of the entire mathematical
exploration by Pharos. In the course of this exploration, Pharos
proposed seven major routes in total and refined them into 30
sub-routes. Pharos assigned these sub-routes to different workers
and, according to the workers' feedback, proposed new sub-routes. For
the
details of these routes, see Table~\ref{tab:routes} and
Table~\ref{tab:timeline}.

Although the final main theorem uses only four facts, this does not
mean that the exploration behind the remaining 269 facts was wasted.
These facts contributed to the final proof indirectly. They are the
results of the broad exploration of the routes, and the experience
they supplied guided the main agent in judging the routes and in
proposing the new sub-routes from which the final result emerged.

At the start, Pharos proposed the seven routes as a coarse framework
of ideas, largely following the techniques that the literature had
applied to the problem, and explored each of them at length: the
sub-routes opened at the start typically ran for 6 to 17 hours and
produced 15 to 34 facts each (Tables \ref{tab:routes}
and~\ref{tab:timeline}). These initial approaches did not by themselves yield the complete proof, but they were explored extensively
before new sub-routes were proposed, and their experiences became the
material for the next stage.

Through the reflection cycle these coarse ideas were gradually
refined into more detailed ones. After several rounds of judging
which routes were still making progress, Pharos described the
possible routes more precisely: the sub-routes opened in the second
half of the run mostly lasted one to three hours and produced one to
five facts each, and a route judged to have stalled, R1, was
stopped. The context design described above is intended to let the
main agent consider the problem as a whole over a run of this length;
here the main agent worked in a single session throughout
(Table~\ref{tab:tokens}), tracked the progress of the routes, and
reassigned workers between them as its assessment changed
(Table~\ref{tab:timeline}), ending with the stop of all workers once
the target was accepted.

The refinement did not only narrow the existing routes. Summarizing
the experiences of the conventional ones led Pharos to a novel
mathematical idea that arose during the run; it was tested on special
models in R3-2 to R3-4 and then extended in R3-6 to the method of
fact C. The fact graph also let
this route reuse the intermediate results of the others: fact B, on
R6, cites fact A from R3, and fact D, on R3, cites fact B from R6
(Table~\ref{tab:routes}), so different workers borrowed each other's
intermediate results across routes.

Taken together, these observations illustrate the combination of mathematical insight and global strategic judgment described above. Context management, supported by helper agents and the route board, enabled the main agent to sustain high-level mathematical reasoning—synthesizing lessons from dozens of sub-routes into new ideas and redirecting the search accordingly. The fact graph kept earlier results available for reuse, allowing the new method to be combined with advances from other routes in the final proof.

The original run described above received no human input beyond the problem
itself. The experts who posed the problem had, however, also
written down four groups of suggestions before Pharos was run, summarizing their preliminary exploration of the problem. These included negative information about routes they had already examined, together with discussion of several directions that appeared potentially worth pursuing. To measure the effect of such advice we ran an ablation run, in which these suggestions were supplied to Pharos
at launch as supplementary material, through the interface with human operators described above. The ablation run took 11 hours and 38 minutes from the creation of the project to the acceptance of the target, and a further 4 hours and 21 minutes to write and check its own manuscript. Its token usage, metered by a ChatGPT Pro subscription, is about two thirds of the weekly usage limit. The final fact graph contains 118 verified facts, of which 4 form the closure of the target fact. The final proof is essentially the same as that of the original run.

None of the suggested directions became part of the final proof. Their effect was instead to change the order of exploration. Pharos adopted them as its initial routes and, within the first hour, found verified counterexamples showing where the obstruction in each direction lies. These tests also produced, at 0:53 and 1:38, two of the four facts in the final closure; in the original run, the corresponding advances appeared at 5:42 and 19:46, respectively. More importantly, the failed routes isolated the remaining obstruction: the positive directions obtained after finite covers had to be integrated. The main agent recorded this gap on the route board at 1:31 and assigned a worker to it at 3:47. That line reached the weighted-jet statement at 10:46, compared with 15:43 in the original run, and reached the final target at 11:38, more than eight hours earlier than the original run. Thus, in this case, the experts' exclusionary guidance accelerated Pharos's path to the solution.

This ablation illustrates one productive mode of interaction between mathematicians and Pharos. In this mode, Pharos can incorporate expert guidance without requiring it to determine the eventual route, testing suggested directions at scale, retaining useful experience from both successful and failed attempts, and redirecting the mathematical exploration accordingly.

\begin{table}[htp]
\centering\footnotesize
\setlength{\tabcolsep}{4pt}
\begin{tabular}{@{}llrlp{0.50\textwidth}@{}}
\hline
Route & Sub-route & Facts & Main-theorem closure & Approach\\
\hline
R1 & \begin{tabular}[t]{@{}l@{}}\\R1-1\\R1-2\end{tabular} & \begin{tabular}[t]{@{}r@{}}28\\2\\26\end{tabular} &  &
Carry the positivity of $-K_X$ on the seed as a Ricci reserve along area-decreasing deformations of the seed disk, accounting for the losses at caps and boundaries, so as to obtain marked disks of bounded area through $f(p)$.\\[4pt]
R2 & \begin{tabular}[t]{@{}l@{}}\\R2-1\\R2-2\\R2-3\\R2-4\\R2-5\\R2-6\end{tabular} & \begin{tabular}[t]{@{}r@{}}57\\32\\15\\5\\2\\2\\1\end{tabular} &  &
Amplify the numerical positivity by finite covers and cuts of the seed curve, then glue the pieces back at a boundary cost, so as to obtain a repeatable descent of marked area.\\[4pt]
R3 & \begin{tabular}[t]{@{}l@{}}\\R3-1\\R3-2\\R3-3\\R3-4\\R3-5\\R3-6\end{tabular} & \begin{tabular}[t]{@{}r@{}}48\\34\\4\\3\\2\\3\\2\end{tabular} & \begin{tabular}[t]{@{}l@{}}\\fact A\\\\\\\\\\facts C, D\end{tabular} &
Realize the positive directions of $f^*T_X$ as a surface germ with positive normal bundle along the seed curve, whose closure is a ruled surface with the seed as a section (tested on special models in R3-2 to R3-4; the method of fact C in R3-6).\\[4pt]
R4 & \begin{tabular}[t]{@{}l@{}}\\R4-1\\R4-2\\R4-3\\R4-4\\R4-5\end{tabular} & \begin{tabular}[t]{@{}r@{}}43\\17\\18\\3\\4\\1\end{tabular} &  &
Take limits of nearly closed positive currents made from disks, separating the retained principal disk from the rational curves that split off.\\[4pt]
R5 & \begin{tabular}[t]{@{}l@{}}\\R5-1\\R5-2\\R5-3\\R5-4\\R5-5\end{tabular} & \begin{tabular}[t]{@{}r@{}}53\\27\\5\\9\\4\\8\end{tabular} &  &
Deform the seed inside a family of embedded curves, the component containing a cover of the seed, out to the boundary of the family, where rational curves of positive area appear.\\[4pt]
R6 & \begin{tabular}[t]{@{}l@{}}\\R6-1\\R6-2\\R6-3\\R6-4\end{tabular} & \begin{tabular}[t]{@{}r@{}}40\\20\\15\\1\\4\end{tabular} & \begin{tabular}[t]{@{}l@{}}\\fact B\\\\\\\end{tabular} &
Control disk boundaries by equal-period fillings in foliated tubes; later, construct an adjoint contraction cohomologically and look for rational curves in its fibers.\\[4pt]
R7 & \begin{tabular}[t]{@{}l@{}}\\R7-1\\R7-2\end{tabular} & \begin{tabular}[t]{@{}r@{}}4\\2\\2\end{tabular} &  &
Finite-area marked compactness and the criteria that extract a marked sphere from a degenerating family, the tools that every other route was meant to feed.\\[4pt]
\hline
\end{tabular}
\vspace{6pt}
\caption{Routes and sub-routes. The ``Facts'' column reports the number of accepted facts produced by each route.}
\label{tab:routes}
\end{table}

\clearpage

\begin{table}[htp]
\centering\footnotesize
\setlength{\tabcolsep}{4pt}
\begin{tabular}{@{}r*{7}{l}@{}}
\hline
Hour & \begin{tabular}[t]{@{}l@{}}W1\\(xhigh)\end{tabular} & \begin{tabular}[t]{@{}l@{}}W2\\(xhigh)\end{tabular} & \begin{tabular}[t]{@{}l@{}}W3\\(xhigh)\end{tabular} & \begin{tabular}[t]{@{}l@{}}W4\\(max)\end{tabular} & \begin{tabular}[t]{@{}l@{}}W5\\(max)\end{tabular} & \begin{tabular}[t]{@{}l@{}}W6\\(max)\end{tabular} & \begin{tabular}[t]{@{}l@{}}W7\\(max)\end{tabular}\\
\hline
 0 & R7-1 & R7-2$\to$R6-1 & R4-1 & R1-1 & R2-1 & R3-1 & R2-2\\
 1 & R7-1$\to$R5-1 & R6-1 & R4-1 & R1-2 & R2-1 & R3-1 \textbf{fact A} & R2-2\\
 2 & R5-1 & R6-1 & R4-1 & R1-2 & R2-1 & R3-1 & R2-2\\
 3 & R5-1 & R6-1 & R4-1 & R1-2 & R2-1 & R3-1 & R2-2\\
 4 & R5-1 & R6-1 & R4-1 & R1-2 & R2-1 & R3-1 & R2-2\\
 5 & R5-1 & R6-1 \textbf{fact B} & R4-1 & R1-2 & R2-1 & R3-1 & R2-2\\
 6 & R5-1 & R6-1 & R4-1 & R1-2 & R2-1 & R3-1 & R2-2$\to$R5-2\\
 7 & R5-1 & R6-1 & R4-1$\to$R4-2 & R1-2 & R2-1 & R3-1 & R5-2\\
 8 & R5-1 & R6-1 & R4-2 & R1-2 & R2-1 & R3-1 & R5-2\\
 9 & R5-1 & R6-1$\to$R6-2 & R4-2 & R1-2 & R2-1 & R3-1 & R5-2$\to$R5-3\\
10 & R5-1 & R6-2 & R4-2 & R1-2 & R2-1 & R3-1 & R5-3\\
11 & R5-1 & R6-2 & R4-2 & R1-2 & R2-1 & R3-1 & R5-3\\
12 & R5-1 & R6-2 & R4-2 & R1-2 & R2-1 & R3-1 & R5-3\\
13 & R5-1 & R6-2 & R4-2 & R1-2 & R2-1$\to$R2-3 & R3-1$\to$R5-4 & R5-3\\
14 & R5-1$\to$R3-2 & R6-2 & R4-2 & R1-2 & R2-3 & R5-4 & R5-3$\to$R5-5\\
15 & R3-2$\to$R3-3 & R6-2 & R4-2$\to$R4-3 & R1-2 & R2-3 & R5-4 & R5-5\\
16 & R3-3 & R6-2$\to$R6-3 & R4-3 & R1-2 & R2-3$\to$R2-4 & R5-4 & R5-5\\
17 & R3-3$\to$R3-4 & R6-4 & R4-4 & R1-2$\to$R3-5 & R2-4 & R5-4 & R5-5\\
18 & R3-4$\to$R3-6 & R6-4 & R4-4 & R3-5 & R2-4 & R5-4 & R5-5\\
19 & R3-6 \textbf{facts C, D} & R6-4 & R4-4$\to$R4-5 & R2-5 & R2-4$\to$R2-6 & R5-4 & R5-5\\
\hline
\end{tabular}
\vspace{6pt}
\caption{Workers by hour after the creation of the project. A cell
names the sub-route of Table~\ref{tab:routes} assigned to the worker.
An arrow indicates a task reassignment.}
\label{tab:timeline}
\end{table}

Beyond finding and checking a proof, we also want it formalized, and
formalized quickly. A paper produced by Pharos has passed the
system's own informal verifier, but not a rigorous check, and before
mathematicians accept it, it must still be checked by humans. A formal
proof of the main statements and of the key steps of the method
reduces this check to the statements themselves: it suffices to
confirm that the formal statements, with the definitions they use, say
what the paper claims. As AI systems produce efficiently, fast and autonomous
formalization helps mathematicians confirm their correctness quickly.
SHEAF is built for this class of tasks.

Mathematicians use systems such as Pharos on frontier problems, and these problems often lie beyond what Mathlib
supports. Moreover, an AI system often uses techniques from fields
other than that of the problem. These features make such papers hard
to formalize and raise two challenges. First, a statement can be formalized only after every
definition it uses has been unfolded down to Mathlib; here a correct
formal statement of the main theorem alone requires about 15{,}600
lines. Second, the proofs invoke standard theories that are absent
from Mathlib and are themselves hard to formalize. Formalizing such a paper quickly is therefore
not a matter of proving harder lemmas, but of deciding what must be
built at all.

The core of SHEAF is the pruning of its dependency graph. SHEAF
unfolds the paper into a directed acyclic graph (DAG) whose nodes are
the statements of the argument, each with a self-contained
natural-language proof, down to results already in Mathlib. Yet the natural-language proof is often
not the easiest route to a formal one: a step may cite a general
standard theorem where a special case suffices, or a route much easier
to formalize may exist. Agents therefore fill in the proofs from the
targets downwards, taking only what each step actually uses; every
edge left unused is deleted, and every node that no longer leads to a
target leaves the work queue. This pruning is safe because the statements of all nodes are
formalized, from Mathlib upwards, and the targets locked before any
proof is attempted. Of the 1{,}675 nodes of the initial DAG, 1{,}036 needed a proof; the
rest were already in Mathlib. Of these 1{,}036 nodes, 745 were proved and 289 were pruned. Without
pruning, those 289 would also have added at least
40\% to the work actually done. The actual saving is far larger,
since the pruned nodes include textbook theorems such as Serre
duality, the resolution of singularities of surfaces and the
Nakai--Moishezon criterion, none of which had to be proved.

The formalization was carried out alongside the authors' revision of
the manuscript and was completed within a week.
Since the paper gives a new proof of a known theorem, the key
intermediate results of its method had to be formalized as well as
the main theorem. These results changed from draft to draft, so the
formalization also produced material that the final text no longer
uses. At completion the
library held about 445{,}000 lines of Lean in about 3{,}000 modules. The final version formalizes the main theorem and eight key
intermediate results of the final text; after duplicates and unused
material were removed, it keeps more than 300{,}000 lines in more than
2{,}000 modules. The run used GPT-6 Astra, GPT-5.6 Sol and Claude
Fable~5.1, at a total API cost of almost US\$50{,}000.
Because resources were limited, about fifteen agents ran concurrently
on average.

\clearpage

\end{document}